\documentclass{amsart}[12pt]

\usepackage{amscd,amssymb}
\usepackage{comment}

\newtheorem{theorem}{Theorem}[section]
\newtheorem{lemma}[theorem]{Lemma}
\newtheorem{proposition}[theorem]{Proposition}
\newtheorem{corollary}[theorem]{Corollary}

\theoremstyle{definition}
\newtheorem{definition}[theorem]{Definition}
\newtheorem{example}[theorem]{Example}

\newtheorem{remark}[theorem]{Remark}

\newcommand{\GL} {\mathrm{GL}}
\newcommand{\SL} {\mathrm{SL}}
\newcommand{\OO} {\mathrm{O}}
\newcommand{\SO} {\mathrm{SO}}
\newcommand{\Sp} {\mathrm{Sp}}

\def\CC {{\mathbb C}}     
\def\NN {{\mathbb N}}     
\def\ZZ {{\mathbb Z}}     

\begin{document}

\title[Symplectic and orthogonal invariants]
{Invariants of symplectic and orthogonal groups acting on $\GL(n,\CC)$-modules}

\author[Vesselin Drensky and Elitza Hristova]
{Vesselin Drensky and Elitza Hristova}
\address{Institute of Mathematics and Informatics,
Bulgarian Academy of Sciences,
Acad. G. Bonchev Str., Block 8,
1113 Sofia, Bulgaria}
\email{drensky@math.bas.bg, e.hristova@math.bas.bg}

\thanks{Partially supported by Project DFNP-86/04.05.2016
of the Young Researchers Program of the Bulgarian Academy of Sciences and
by Grant KP-06 N 32/1 of 07.12.2019 of the Bulgarian National Science Fund}



\begin{abstract}
Let $\GL(n) = \GL(n, \CC)$ denote the complex general linear group and let $G \subset \GL(n)$
be one of the classical complex subgroups $\OO(n)$, $\SO(n)$, and $\Sp(2k)$ (in the case $n = 2k$).
We take a finite dimensional polynomial $\GL(n)$-module $W$ and consider the symmetric algebra $S(W)$.
Extending previous results for $G=\SL(n)$, we develop a method for determining the Hilbert series $H(S(W)^G, t)$ of the algebra of invariants $S(W)^G$.
Then we give explicit examples for computing $H(S(W)^G, t)$.
As a further application, we extend our method to compute also the Hilbert series of the algebras of invariants
$\Lambda(S^2 V)^G$ and $\Lambda(\Lambda^2 V)^G$, where $V = \CC^n$ denotes the standard $\GL(n)$-module.
\end{abstract}

\maketitle

\section{Introduction} \label{sec_Intro}

Let $\GL(n) = \GL(n, \CC)$ be the general linear group with its canonical action on the $n$-dimensional complex vector space $V=\CC^n$ and
let $W$ be a finite dimensional polynomial $\GL(n)$-module. Then $W$ can be written as a direct sum of its irreducible components
\begin{equation} \label{decomposition of W}
W = \bigoplus_{\lambda} k(\lambda) V^n_{\lambda},
\end{equation}
where $\lambda = (\lambda_1, \ldots, \lambda_n) \in \NN_0^n$, $\lambda_1\geq \lambda_2 \geq \cdots \geq \lambda_n \geq 0$,
is a non-negative integer partition and
$V^n_{\lambda}$ is the irreducible $\GL(n)$-module with highest weight $\lambda$.
(In particular, $V=V^n_{(1)}$.) We consider the symmetric algebra
\[
S(W) = \bigoplus_{i \geq 0} S^iW,
\]
where $S^iW$ denotes the $i$-th symmetric power of $W$. Then $\GL(n)$ and its subgroups act canonically on $S(W)$ by the usual diagonal action
and we can construct the algebra of invariants $S(W)^G$, where $G$ is a subgroup of $\GL(n)$.
In classical invariant theory one considers also the algebra of polynomial functions $\CC[W]$.
The group $\GL(n)$ and its subgroups $G$ act canonically on $\CC[W]$ by the formula
\[
(gf)(v) = f(g^{-1}v)\text{ for all }v \in W\text{ and }f \in \CC[W].
\]
One uses this action and studies the algebras of invariants $\CC[W]^G$,
where again $G$ is a subgroup of $\GL(n)$. 


We recall the following definition.
\begin{definition}
Let $\displaystyle A = \bigoplus_{i\geq 0} A^i$ be a finitely generated graded (commutative or non-commutative) algebra over $\CC$
such that $A^0=\CC$ or $A^0=0$.
The {\bf Hilbert series} of $A$ is the formal power series
\[
H(A, t) = \sum_{i\geq 0} (\dim A^i) t^i.
\]
\end{definition}
The Hilbert series $H(A,t)$ is one of the most important invariants of the graded algebra.
In particular, when we consider a minimal set of generators of $A$, $H(A,t)$ gives information about the lowest degree of the generators in this set and the maximal number of generators in each degree.

Both algebras $\CC[W]^G$ and $S(W)^G$ have a natural $\NN_0$-grading which is inherited, respectively,
from the $\NN_0$-gradings of $\CC[W]$ and $S(W)$.
Furthermore, $\CC[W]^G$ and $S(W)^G$ for $G=\OO(n)$, $\SO(n)$, $\Sp(2k)$ are isomorphic as $\NN_0$-graded algebras and hence
$H(\CC[W]^G, t) = H(S(W)^G, t)$. In the sequel we shall work in $S(W)$. 

There are many methods to compute the Hilbert series $H(\CC[W]^G, t)$ (see, e.g., \cite{DK}).
In a series of joint papers of the first named author (see \cite{BBDGK} for an account), one more method
for computing the Hilbert series $H(S(W)^{\SL(n)}, t)$
of the algebra of invariants $S(W)^{\SL(n)}$ has been developed. It is based on the method of Elliott \cite{E} from 1903
for finding the non-negative solutions of linear systems of homogeneous Diophantine equations, further developed by MacMahon \cite{MM}
in his $\Omega$-calculus (or Partition Analysis), and combined with the approach of Berele \cite{B} in the study of cocharacters
of algebras with polynomial identities.
Our goal in this paper is to extend the latter method and to determine also the Hilbert series of the algebras of invariants
$S(W)^G$ for $G = \OO(n)$, $G = \SO(n)$, or $G = \Sp(2k)$.
Our main results are given in Section \ref{sec_HilbSeries}.

In Section \ref{sec_Examples}, using our results from Section \ref{sec_HilbSeries},
we compute the Hilbert series of $S(W)^G$ for explicit examples of $W$.
Some of the examples are classical and in these examples the algebra of invariants $S(W)^G$ is already described in the literature in terms of generators and relations, but there are several low dimensional examples, which we believe are new. In some examples we also explore the relations between the algebras $S(W)^{\OO(n)}$ and $S(W)^{\SO(n)}$.

In Sections \ref{sec_S} and \ref{sec_Lambda}, as a further application of our method, we compute also the Hilbert series of
the algebras of invariants $\Lambda(S^2 V)^G$ and $\Lambda(\Lambda^2 V)^G$ for $G = \OO(n)$, $\SO(n)$, $\Sp(2k)$,
where $\Lambda(W)$ and $\Lambda^2(W)$ denote, respectively,
the exterior algebra and the second exterior power of the $\GL(n)$-module $W$. 


\section{Decomposition of irreducible $\GL(2k)$-modules over $\Sp(2k)$} \label{secBR_Sp}
In this section $n=2k$.  By $V^{2k}_{\lambda}$ we denote again the irreducible $\GL(2k)$-module with highest weight $\lambda$.
Our goal is to decompose $V^{2k}_{\lambda}$ as a module over $\Sp(2k)$ and to determine the dimension of the subspace of invariants $(V^{2k}_{\lambda})^{\Sp(2k)}$.
The irreducible representations of $\Sp(2k)$ are indexed by non-negative integer partitions $\mu$ with at most $k$ parts,
i.e., $\mu = (\mu_1, \ldots, \mu_k, 0, \ldots, 0)$ (see, e.g., \cite{FH, GW}). We denote them by $V^{2k}_{\left\langle \mu \right\rangle}$.
For a partition $\lambda = (\lambda_1,\dots, \lambda_n)$ we write $\lambda' = (\lambda_1', \dots, \lambda_n')$ for the transpose and by $2\delta = (2\delta_1,\ldots, 2\delta_n)$ we denote an even partition. With these notations the following Littlewood-Richardson branching rule holds.

\begin{proposition}\cite{HTW, K} \label{prop_LRruleSp}
Let $\lambda$ be a partition in at most $k$ parts. Then
\begin{equation}\label{eq_LRruleSp}
{V^{2k}_{\lambda}}{\downarrow \Sp(2k)} \cong \bigoplus_{\mu, 2\delta} c^{\lambda}_{\mu (2\delta)'}V^{2k}_{\left\langle \mu \right\rangle},
\end{equation}
where the sum runs over all partitions $\mu$
and all even partitions $2\delta$.
Here the coefficients 
$c^{\lambda}_{\mu \nu}$ are the Littlewood-Richardson coefficients.
\end{proposition}

Since in the statement of Proposition \ref{prop_LRruleSp} $\lambda$ is a partition in at most $k$ parts, the same holds for the partitions $\mu$.
Hence we do obtain the decomposition of $V^{2k}_{\lambda}$ into a direct sum of irreducible $\Sp(2k)$-modules
$V^{2k}_{\left\langle \mu \right\rangle}$.
When the partition $\lambda$ has more than $k$ parts,
then on the right side of the equation (\ref{eq_LRruleSp}) there will appear terms $V^{2k}_{\left\langle \mu \right\rangle}$,
for which $\mu$ has more that $k$ parts. 
In the paper \cite{K}, it is shown how to regard such terms as elements of the Grothendieck group of $\Sp(2k)$-modules with the help of modification rules.
For the group $\Sp(2k)$ the modification rule is as follows.
Let $\mu = (p, \mu_2', \ldots, \mu_q')'$, i.e., let $\mu$ have $p$ rows with $p > k$. Then the following equivalence formula is derived in \cite{K}:
\begin{align} \label{eq_Sp}
V^{2k}_{\left\langle \mu \right\rangle} = (-1)^{x+1} V^{2k}_{\left\langle \sigma \right\rangle},
\end{align}
where the Young diagram of $\sigma$ is obtained from the Young diagram of $\mu$ by the removal of a continuous boundary hook of length $2p-n-2$
starting from the bottom box of the first column of the Young diagram of $\mu$.
Here $x$ denotes the depth of the hook, i.e., $x+1$ is the number of columns in the hook.
We then repeat this process of removal of a continuous boundary hook until we obtain an admissible Young diagram, i.e. a Young diagram corresponding to a partition $\mu$ with at most $k$ parts.
We write zero for the multiplicity of $V^{2k}_{\left\langle \mu \right\rangle}$ in the branching formula
if the process stops before we obtain an admissible Young diagram (because $2p-n-2=0$)
or we obtain a configuration of boxes which is not a Young diagram.
The latter happens if for the columns of the configuration corresponding to
$\sigma$ the rule $\sigma'_1 \geq \sigma'_2 \geq \cdots \geq \sigma'_n$ is violated.


\begin{proposition}\label{inadmissible reps}
Let $\mu$ be a partition with more than $k$ parts and let $V^{2k}_{\left\langle \mu \right\rangle}$ be the element of the Grothendieck group of $\Sp(2k)$-modules defined by formula (\ref{eq_Sp}). Then $V^{2k}_{\left\langle \mu \right\rangle}$ is equivalent neither to the trivial one-dimensional $\Sp(2k)$-module $V^{2k}_{\left\langle (0,\ldots,0) \right\rangle}$, nor to its inverse in the Grothendieck group of $\Sp(2k)$-modules.
\end{proposition}

\begin{proof}
Let $V^{2k}_{\left\langle \mu \right\rangle}$ be as in the statement of the proposition. Let us assume that,
starting with the Young diagram of the partition $\mu$ and removing
continuous boundary hooks, we obtain the admissible Young diagram without boxes corresponding to the partition $(0,\ldots,0)$.
Hence, one step before the end of the process we shall reach a partition $\nu=(\nu_1,1,\ldots,1)$ with $p>k$ parts.
Since $\nu$ will disappear in the next step, its Young diagram has exactly $2p-2k-2$ boxes, i.e., $\nu_1+p-1=2p-2k-2$
and $\nu_1=p-2k-1=p-n-1<0$ because $\nu$ is a partition in $p\leq n$ parts. This contradiction shows that
$V^{2k}_{\left\langle \mu \right\rangle}$ cannot be equivalent to the trivial
one-dimensional $\Sp(2k)$-module $V^{2k}_{\left\langle (0,\ldots,0) \right\rangle}$.
\end{proof}

\begin{corollary} \label{coro_Sp}

Let ${V^{2k}_{\lambda}}$ be any irreducible $\GL(2k)$-module. Then
\begin{align*}
\dim (V^{2k}_{\lambda})^{\Sp(2k)} = \left \{ \begin{array} {ll}
                                                                                    1 & \text{ if } \lambda_1 = \lambda_2, \lambda_3 = \lambda_4, \dots, \lambda_{2k-1} = \lambda_{2k} \\
                                                                                    0 & \text{ otherwise.} \end{array}
                                                                \right.
\end{align*}
\end{corollary}

\begin{proof}
By Propositions \ref{prop_LRruleSp} and \ref{inadmissible reps}, and the modification rules stated between them,
it is sufficient to calculate in (\ref{eq_LRruleSp})
the Littlewood-Richardson coefficient $c^{\lambda}_{\mu (2\delta)'}$ for the partition $\mu=(0,\ldots,0)$ and to show that
\[
c^{\lambda}_{\mu (2\delta)'}
=\begin{cases} 1, \text{ if }\lambda=(2\delta)'\\
0, \text{ otherwise.}\\
\end{cases}
\]
In order to calculate $c^{\lambda}_{\mu (2\delta)'}$, we start with the diagram of $(2\delta)'$ and add to it the boxes of the diagram of $\mu$
to obtain the diagram of $\lambda$ filling in the boxes from $\mu$ with integers following the Littlewood-Richardson rule (see, e.g., \cite{M}).
Then $c^{\lambda}_{\mu (2\delta)'}$
is equal to the possible ways to do these fillings in. Since the diagram of $\mu$ has no boxes, the only diagram we obtain, and
exactly once, is the diagram of $(2\delta)'$, i.e.,
$c^{\lambda}_{\mu (2\delta)'}=1$ for $\lambda=(2\delta)'$ and $c^{\lambda}_{\mu (2\delta)'}=0$ otherwise.
Clearly, $\lambda=(2\delta)'$ means that $\lambda_1 = \lambda_2$, $\lambda_3 = \lambda_4$, \ldots, $\lambda_{2k-1} = \lambda_{2k}$.
\end{proof}

\begin{remark} The statement of Corollary \ref{coro_Sp} is given with a different proof in \cite{P} on page 410.
\end{remark}

\section{Decomposition of irreducible $\GL(n)$-modules over $\OO(n)$ and $\SO(n)$} \label{secBR_OSO}
In this section we determine the dimensions of the subspaces of $\OO(n)$ and $\SO(n)$ invariants $(V^n_{\lambda})^{\OO(n)}$ and $(V^n_{\lambda})^{\SO(n)}$.
We use a similar approach as in Section \ref{secBR_Sp}.
We start with the description of the structure of $V^n_{\lambda}$ as an $\OO(n)$-module.
The irreducible representations of $\OO(n)$ are indexed by partitions $\mu$ with $\mu'_1 + \mu'_2 \leq n$,
i.e., the sum of the lengths of the first two columns of the Young diagram of $\mu$ should be at most $n$ (see, e.g., \cite{FH, GW}).
We denote the corresponding $\OO(n)$-modules by $V^n_{[\mu]}$. With these notations the following Littlewood-Richardson rule holds.

\begin{proposition}\cite{HTW, K} \label{prop_LRruleO}
Let $\lambda$ be a partition with at most $n/2$ parts. Then
\begin{align*}
{V^{n}_{\lambda}}{\downarrow \OO(n)} \cong \bigoplus_{\mu, 2\delta} c^{\lambda}_{\mu (2\delta)}V^{n}_{[\mu]},
\end{align*}
where the sum runs over all partitions $\mu$ 
and all even partitions $2\delta = (2\delta_1,\ldots, 2\delta_n)$.
\end{proposition}

When $\lambda$ has more than $n/2$ parts the above branching formula does not hold and we use again the modification rules from \cite{K}.
The Young diagram of a parition $\mu = (p, \mu_2', \ldots, \mu_q')'$ is called inadmissible for $\OO(2k)$ or $\OO(2k+1)$ if $p > k$, i.e., if the first column in the Young diagram of $\mu$ has more than $k$ boxes. For such $\mu$, we regard the term $V^{n}_{[\mu]}$ as an element of the Grothendieck group of $\OO(n)$-modules using the following equivalence formula from \cite{K}:
\begin{equation} \label{eq_Equivalence_O}
V^{n}_{[\mu]} = (-1)^x \varepsilon V^{n}_{[\sigma]},
\end{equation}
where the Young diagram of $\sigma$ is obtained from the Young diagram of $\mu$ by the removal of a continuous boundary hook of length $2p - n$
starting in the first column of $\mu$. Here again $x$ denotes the depth of the hook and $\varepsilon$
is the determinant of the matrix of the particular group element from $\OO(n)$ acting on $V^{n}_{\lambda}$.
As before, we repeat the process of removal of a continuous boundary hook until we obtain an admissible Young diagram
or until we obtain a configuration of boxes which is not a Young diagram.
In the latter case we write zero in the branching formula. Using this modification rule and repeating the arguments in the proof
of Proposition \ref{inadmissible reps} we obtain the following statement.


\begin{proposition}Let $\mu$ be a partition with more than $\left\lfloor \frac{n}{2}\right\rfloor$ parts and let $V^{n}_{[\mu]}$ be the element of the Grothendieck group of $\OO(n)$-modules defined by formula (\ref{eq_Equivalence_O}). Then $V^{n}_{[\mu]}$ is equivalent neither to the trivial one-dimensional $\OO(n)$-module, nor to its inverse in the Grothendieck group of $\OO(n)$-modules.
\end{proposition}

Then, as in Section \ref{secBR_Sp}, we obtain the description of the $\GL(n)$-modules ${V^{n}_{\lambda}}$ which contain the trivial $\OO(n)$-module
$V^{n}_{[(0,\ldots,0)]}$.

\begin{corollary} \label{coro_O} Let ${V^{n}_{\lambda}}$ be any irreducible $\GL(n)$-module. Then
\begin{align*}
&\dim (V^n_{\lambda})^{\OO(n)} = \left \{ \begin{array} {ll}
                                                                                    1 & \text{ if } \lambda \text{ is an even partition}\\
                                                                                    0 & \text{ otherwise.} \end{array}
                                                                \right.
\end{align*}
\end{corollary}

\begin{remark} The statement of Corollary \ref{coro_O} can be found again, using a different approach, in \cite{P} on page 409.
\end{remark}

When we consider the subgroup $\SO(n)$, then $\varepsilon = 1$ for any group element in the equivalence formula (\ref{eq_Equivalence_O}).
Furthermore, all irreducible $\OO(n)$-modules $V^{n}_{[\mu]}$ remain irreducible when restricted to $\SO(n)$,
except for the case $n = 2k$ and $\mu = (\mu_1, \ldots, \mu_k, 0, \ldots, 0)$ with $\mu_k \neq 0$.
Such representations split into two irreducible $\SO(n)$-representations. Using these considerations we make the following observation.


\begin{proposition} \label{prop_Triv}
Let $\mu$ be a partition with more than $\left\lfloor \frac{n}{2}\right\rfloor$ parts and let $V^{n}_{[\mu]}$ be the element of the Grothendieck group of $\SO(n)$-modules defined by formula (\ref{eq_Equivalence_O}). Then $V^{n}_{[\mu]}$ is equivalent to the trivial one-dimensional $\SO(n)$-module if and only if $\mu = \underbrace{(1,1,\ldots, 1)}_n$.
\end{proposition}

\begin{corollary} \label{coro_SO} Let ${V^{n}_{\lambda}}$ be any irreducible $\GL(n)$-module. Then
\begin{align*}
&\dim (V^n_{\lambda})^{\SO(n)} = \left \{ \begin{array} {ll}
                                                                                    1 & \text{ if } \lambda \text{ is an even or an odd partition}\\
                                                                                    0 & \text{ otherwise.} \end{array}
                                                                \right.
\end{align*}
\end{corollary}

\begin{proof}
In view of Propositions \ref{prop_LRruleO} and \ref{prop_Triv} we only need to evaluate the Littlewood-Richardson
coefficient $c^{\lambda}_{\mu (2\delta)}$ for $\mu = (0, \ldots, 0)$ and $\mu = \underbrace{(1,1,\ldots, 1)}_n$. For $\mu = (0, \ldots, 0)$ this is trivial.

When $\mu = (1,1,\ldots, 1)$ we use the following Pieri rule (see, e.g., \cite{FH}): $c^{\lambda}_{\mu (2\delta)} = 1$
if and only if we can obtain $\lambda$ from $2\delta$ by adding one box to each row.
In all other cases $c^{\lambda}_{\mu (2\delta)} = 0$. In other words, the only possibility for $\lambda$ is
$\lambda = (2\delta_1 + 1, \ldots, 2\delta_n +1)$. Thus the statement follows.
\end{proof}

\section{Determining the Hilbert series} \label{sec_HilbSeries}
The goal of this section is to determine the Hilbert series $H(S(W)^G, t)$ for $G = \OO(n)$, $\SO(n)$, and $\Sp(2k)$
by using Hilbert series of multigraded algebras.
We recall that if
\[
A = \bigoplus_{\mu \in \NN_0^n} A(\mu)
\]
is a finitely generated algebra with an $\NN_0^n$-grading,
then the Hilbert series of $A$ with respect to this grading is the formal power series $H(A, x_1, \ldots, x_n) \in \ZZ[[x_1, \dots, x_n]]$ defined by
\[
H(A, x_1, \ldots, x_n) = \sum_{\mu = (\mu_1, \ldots, \mu_n) \in \NN_0^n} \dim A(\mu) x_1^{\mu_1} \cdots x_n^{\mu_n}.
\]
This definition makes sense also for multigraded vector spaces. One example of a vector space with an $\NN_0^n$-grading is the $\GL(n)$-module $V^n_{\lambda}$ together with its weight space decomposition.
The Hilbert series of $V^n_{\lambda}$ with respect to this grading has the form
\[
H(V^n_{\lambda}, x_1, \ldots, x_n) = S_{\lambda}(x_1, \ldots, x_n),
\]
where $S_{\lambda}(x_1, \ldots, x_n)$ is the Schur polynomial corresponding to the partition $\lambda$.
Consequently, any polynomial $\GL(n)$-module $W$ has an $\NN_0^n$-grading and a corresponding Hilbert series which is again expressed via Schur polynomials.

Let $W$ be any polynomial $\GL(n)$-module. We take the decomposition of the symmetric algebra $S(W)$ into irreducible $\GL(n)$-modules
\begin{equation*} 
S(W) = \bigoplus_{l\geq 0} S^lW =  \bigoplus_{l\geq 0}\bigoplus_{\lambda} m_l(\lambda) V^n_{\lambda},
\end{equation*}
where the second sum runs over all partitions $\lambda \in \NN_0^n$.
Thus $S(W)$ possesses a natural $\NN_0$-grading coming from the decomposition into homogeneous components
and a natural $\NN_0^n$-grading coming from the weight space decomposition of each $V^n_{\lambda}$.
As in \cite{BBDGK}, we consider the following Hilbert series of $S(W)$, which takes into account both gradings
\begin{align*}
H(S(W); x_1, \ldots, x_n, t) = &\sum_{l \geq 0} H(S^lW, x_1, \ldots, x_n) t^l = \\ &\sum_{l \geq 0}\left (\sum_{\lambda} m_l(\lambda) S_{\lambda}(x_1, \ldots, x_n)\right ) t^l.
\end{align*}
Clearly, $H(S(W); x_1,\ldots, x_n, t) \in \ZZ[[x_1, \ldots, x_n]][[t]]$.
Furthermore, as in \cite{BBDGK} again, we introduce the following two multiplicity series of $H(S(W); x_1, \ldots, x_n, t)$:
\[
M(H(S(W)); x_1, \ldots, x_n, t) = \sum_{l \geq 0}\left(\sum_{\lambda} m_l(\lambda)x_1^{\lambda_1} \cdots x_n^{\lambda_n} \right)t^l,
\]
\[
M'(H(S(W)); v_1,\ldots, v_n, t) = \sum_{l \geq 0}\left(\sum_{\lambda}
m_l(\lambda)v_1^{\lambda_1 - \lambda_2} v_2^{\lambda_2 - \lambda_3} \cdots v_{n-1}^{\lambda_{n-1} - \lambda_n} v_n^{\lambda_n}\right)t^l.
\]
The second multiplicity series is obtained from the first one using the change of variables
\[
v_1 = x_1, v_2 = x_1x_2, \ldots, v_n = x_1 \cdots x_n.
\]
The following theorem is the main tool to calculate the Hilbert series of the algebras of invariants $S(W)^G$ for $G=\Sp(2k),\OO(n),\SO(n)$,
which is done in the next sections:

\begin{theorem} \label{thm_HilbertSeries} Let $W$ be as above.

{\rm (i)} The Hilbert series of the algebra of invariants $S(W)^{\Sp(2k)}$ (where $n = 2k$) is given by
\[
H(S(W)^{\Sp(2k)}, t) = M'(H(S(W)); 0, 1, 0, 1 ,\ldots, 0, 1, t).
\]

{\rm (ii)} The Hilbert series of the algebra of invariants $S(W)^{\OO(n)}$ is
\[
H(S(W)^{\OO(n)}, t) = M_{n}(t),
\]
where $M_n$ is defined iteratively in the following way:
\[
M_1(x_2, \ldots, x_n, t) = \frac{1}{2}\left(M(H(S(W)); -1, x_2, \ldots, x_n,t) + M(H(S(W)); 1, x_2, \ldots, x_n, t)\right),
\]
\[
M_2(x_3, \ldots, x_n, t) = \frac{1}{2}(M_1(-1, x_3, \ldots, x_n, t) + M_1(1, x_3, \ldots, x_n, t))
\]
\[
\cdots\cdots\cdots
\]
\[
M_n(t) = \frac{1}{2}(M_{n-1}(-1, t) + M_{n-1}(1, t)).
\]

{\rm (iii)} The Hilbert series of the algebra of invariants $S(W)^{\SO(n)}$ is
\[
H(S(W)^{\SO(n)}, t) = M'_{n}(t),
\]
where
\[
M'_1(v_2, \ldots, v_n, t) = \frac{1}{2}(M'(H(S(W)); -1, v_2, \ldots, v_n,t) + M'(H(S(W)); 1, v_2, \ldots, v_n, t)),
\]
\[
M'_2(v_3, \ldots, v_n, t) = \frac{1}{2}(M'_1(-1, v_3, \ldots, v_n, t) + M'_1(1, v_3, \ldots, v_n, t))
\]
\[
\cdots\cdots\cdots
\]
\[
M'_{n-1}(v_n, t) = \frac{1}{2}(M'_{n-2}(-1,v_n, t) + M'_1(1, v_n, t)),
\]
\[
M'_n(t) = M'_{n-1}(1, t).
\]
\end{theorem}

\begin{proof}

(i) 
We take again the decomposition of $S(W)$ into irreducible $\GL(n)$-modules:
\begin{equation} \label{eq_PolyAlgDecomp}
S(W) = \bigoplus_{l\geq 0} S^lW =  \bigoplus_{l\geq 0}\bigoplus_{\lambda} m_l(\lambda) V^n_{\lambda},
\end{equation}
where the second sum runs over all partitions $\lambda \in \NN_0^n$.
Therefore,
\[
S(W)^{\Sp(2k)} = \bigoplus_{l\geq 0} \bigoplus_{\lambda} m_l(\lambda) (V^n_{\lambda})^{\Sp(2k)}.
\]
The definition of Hilbert series gives
\[
H(S(W)^{\Sp(2k)}, t) =\sum_{l\geq 0}\left(\sum_{\lambda} m_l(\lambda) \dim (V^n_{\lambda})^{\Sp(2k)}\right) t^l.
\]
Hence, Corollary \ref{coro_Sp} implies that
\[
H(S(W)^{\Sp(2k)}, t) = \sum_{l\geq 0} \left(\sum_{\lambda_1 = \lambda_2, \ldots, \lambda_{2k-1} = \lambda_{2k}}
m_l(\lambda)\right) t^l.
\]

Moreover, if we evaluate the multiplicity series $M'(H(S(W)); v_1,\ldots, v_{2k}, t)$
at the point $(v_1, v_2 \ldots,v_{2k-1}, v_{2k}) = (0,1, \ldots, 0,1)$ we obtain
\[
M'(H(S(W)); 0, 1,\ldots, 0, 1, t) = \sum_{l\geq 0} \left(\sum_{\lambda_1 = \lambda_2, \ldots, \lambda_{2k-1} = \lambda_{2k}}
m_l(\lambda)\right) t^l =
\]
\[
H(S(W)^{\Sp(2k)}, t).
\]

(ii) Similarly, for $H(S(W)^{\OO(n)}, t)$ we obtain
\[
H(S(W)^{\OO(n)}, t) = \sum_{l\geq 0}\left(\sum_{\lambda} m_l(\lambda) \dim (V^n_{\lambda})^{\OO(n)}\right) t^l.
\]
Thus, Corollary \ref{coro_O} implies
\[
H(S(W)^{\OO(n)}, t) = \sum_{l\geq 0} \left(\sum_{\substack{\lambda- \text{an even} \\ \text{partition}}} m_l(\lambda)\right) t^l
= M_{n}(t),
\]
where $M_n$ is defined as in the statement of Theorem \ref{thm_HilbertSeries} (ii).

(iii) For $S(W)^{\SO(n)}$ we obtain, using Corollary \ref{coro_SO}
\[
H(S(W)^{\SO(n)}, t)
= \sum_{l\geq 0} \left(\sum_{\substack{\lambda- \text{an even or an odd} \\ \text{partition}}} m_l(\lambda)\right) t^l
= M'_{n}(t),
\]
where $M'_n$ is defined as in the statement of Theorem \ref{thm_HilbertSeries} (iii).
\end{proof}

\section{Examples} \label{sec_Examples}
In this section we use our results from Section \ref{sec_HilbSeries} to compute the Hilbert series
$H(S(W)^{\Sp(2k)}, t)$, $H(S(W)^{\OO(n)}, t)$, and $H(S(W)^{\SO(n)}, t)$
for several explicit $\GL(n)$-modules $W$.
Using the results we obtain, we can determine in each case
the minimal degree of the generators of the respective algebras of invariants and in particular the number of generators of lowest degree.
However, in general, we cannot give complete information about how many generators there are and in what degrees and we cannot determine
if there are relations between them. Nevertheless, the knowledge of the Hilbert series simplifies the solution of the problem
to find a presentation of the algebra in terms of generators and defining relations (see, e.g., \cite{DK}).
We point out when, to our knowledge, the respective example is already known in the literature,
and in particular when the algebra $S(W)^G$ is a polynomial algebra. 

In the cases we consider below the input data is the decomposition
(\ref{decomposition of W}) of the $\GL(n)$-module $W$ into a sum of irreducible submodules. Then the Hilbert series of $W$ is
\[
H(W,x_1,\ldots,x_n)=\sum_{\lambda} k(\lambda) S_{\lambda}(x_1,\ldots,x_n)=\sum_{\mu=(\mu_1,\ldots,\mu_n)\in\NN_0^n}a_{\mu}x_1^{\mu_1}\cdots x_n^{\mu_n}
\]
and the Hilbert series of $S(W)$ as a multigraded algebra is
\begin{align} \label{eq_HilbSeriesS}
H(S(W),x_1,\ldots,x_n)=\prod_{\mu=(\mu_1,\ldots,\mu_n)\in\NN_0^n}\frac{1}{(1-x_1^{\mu_1}\cdots x_n^{\mu_n})^{a_{\mu}}}.
\end{align}

In the first three examples we know in advance the decomposition of $S(W)$ into irreducible $\GL(n)$-modules,
described in (\ref{eq_PolyAlgDecomp}). Hence, we can easily determine the multiplicity series $M$ and $M'$.
In all three cases the algebra of invariants $S(W)^G$, for $G = \SO(n)$ or $\Sp(2k)$,
is known to be a polynomial algebra and the degrees of its generators are given in \cite{Ka} for irreducible $W$ and in \cite{S} for reducible $W$.
We conclude that $S(W)^{\OO(n)}$ is also a polynomial algebra in these three cases and the expressions
for the respective Hilbert series which we find below are enough to describe the relations between the algebras $S(W)^{\OO(n)}$ and $S(W)^{\SO(n)}$  and, in particular, the degrees of all generators in a minimal set of generators of $S(W)^{\OO(n)}$.

\begin{example} \label{ex_SS2V}
Let $V = \CC^n$ denote the standard $\GL(n)$-module and let $W = S^2V$ be the second symmetric power of $V$.
In other words, $W = V^n_{\lambda}$ with $\lambda = (2, 0, \ldots, 0)$. The decomposition (\ref{eq_PolyAlgDecomp}) is known in this case
(see, e.g., \cite{GW}) and is given by
\[
S(S^2V) = \bigoplus_{l \geq 0}\bigoplus_{\substack{|\lambda| = 2l\\ \lambda- \text{even}}} V^n_{\lambda}.
\]
Thus the multiplicity series
$M(H(S(S^2V)); x_1, \ldots, x_n, t)$  and $M'(H(S(S^2V)); v_1, \ldots, v_n, t)$ are, respectively, equal to (see also \cite{BBDGK}):
\begin{align*}
&M(H(S(S^2V)); x_1, \ldots, x_n, t) = \prod_{i = 1}^n \frac{1}{1-(x_1 \cdots x_i)^2t^i},\\
&M'(H(S(S^2V)); v_1, \ldots, v_n, t) = \prod_{i = 1}^n \frac{1}{1-v_i^2t^i}.
\end{align*}
Using Theorem \ref{thm_HilbertSeries}, we obtain
\begin{align*} 
H(S(S^2V)^{\Sp(2k)}, t) = \prod_{i = 1}^k \frac{1}{1-t^{2i}},\text{ where }n=2k,
\end{align*}
\begin{equation} \label{eq_HilbOSO}
H(S(S^2V)^{\OO(n)}, t) = H(S(S^2V)^{\SO(n)}, t) =\prod_{i = 1}^n \frac{1}{1-t^i}.
\end{equation}

In \cite{Ka} it is shown that the algebra of invariants $S(S^2V)^{\Sp(2k)}$
is a polynomial algebra with $k$ generators in degrees respectively $2,4, \ldots, 2k$.

For $S(S^2V)^{\SO(n)}$ we use, in the notations of Section \ref{secBR_OSO}, that
\[
S^2V \downarrow \SO(n) \cong V^n_{[(2, 0, \ldots, 0)]} \oplus V^n_{[(0, 0, \ldots, 0)]}.
\]
Hence, it follows from \cite{Ka} that $S(S^2V)^{\SO(n)}$ is a polynomial algebra too with $n$ generators in degrees respectively $1,2, \ldots, n$.
\end{example}

Formula (\ref{eq_HilbOSO}) implies the following immediate corollary.

\begin{corollary} For all $n$ we have
$$
S(S^2V)^{\OO(n)} = S(S^2V)^{\SO(n)}.
$$
\end{corollary}
\begin{proof}
This follows from the fact that for each $l \geq 0$ we have $S^l(S^2V)^{\OO(n)} \subseteq S^l(S^2V)^{\SO(n)}$ and at the same time
$
\dim S^l(S^2V)^{\OO(n)} = \dim S^l(S^2V)^{\SO(n)}.
$

\end{proof}

\begin{example} \label{ex_SLambda}
Next, let us take $W = \Lambda^2 V$, the second exterior power of $V$. Then $W = V^n_{\lambda}$ with $\lambda = (1,1, 0, \ldots, 0)$ and it is known that
\[
S(\Lambda^2 V) = \bigoplus_{l \geq 0}\bigoplus_{\substack{\lambda}} V^n_{\lambda},
\]
where the second sum runs over all partitions $\lambda$ with $|\lambda| = 2l$
and such that $\lambda_{2i-1} = \lambda_{2i}$ for $i = 1, \ldots, \lfloor n/2 \rfloor$.
When $n$ is odd we also have $\lambda_{n} = 0$ (see, e.g., \cite{GW}).
Then for the multiplicity series $M$ and $M'$ one obtains (see also \cite{BBDGK})
\begin{align*}
&M(H(S(\Lambda^2V)); x_1, \ldots, x_n, t) = \prod_{i = 1}^{\lfloor \frac{n}{2} \rfloor} \frac{1}{1-(x_1 \cdots x_{2i})t^i},\\
&M'(H(S(\Lambda^2V)); v_1, \ldots, v_n, t) = \prod_{i = 1}^{\lfloor \frac{n}{2} \rfloor} \frac{1}{1-v_{2i}t^i}.
\end{align*}
Thus,
\[
H(S(\Lambda^2V)^{\Sp(2k)}, t) = \prod_{i = 1}^k \frac{1}{1-t^{i}},
\]
\[
H(S(\Lambda^2V)^{\OO(n)}, t) = \prod_{i = 1}^{\lfloor \frac{n}{2} \rfloor} \frac{1}{1-t^{2i}},
\]
\[
H(S(\Lambda^2V)^{\SO(n)}, t) =\begin{cases}
{\displaystyle\frac{1}{1-t^k}\prod_{i = 1}^{k-1} \frac{1}{1-t^{2i}},}   & \text{ if } n = 2k, \\
{\displaystyle\prod_{i = 1}^k \frac{1}{1-t^{2i}},} & \text{ if } n = 2k+1.\end{cases}
\]

It is known (see, e.g., \cite{Ka}) that the algebra of invariants $S(\Lambda^2V)^G$ for $G = \SO(n)$ is a polynomial algebra with the following generators:
if $n = 2k +1$ there are $k$ generators in degrees respectively $2,4,\ldots, 2k$;
if $n = 2k$ there are $k-1$ generators in degrees respectively $2,4,\ldots, 2(k-1)$ and one generator in degree $k$.

For $S(\Lambda^2V)^{\Sp(2k)}$ we use, in the notations of Section \ref{secBR_Sp} that
\[
\Lambda^2V \downarrow \Sp(2k) \cong V^n_{\left\langle (1, 1,0, \ldots, 0)\right\rangle} \oplus V^n_{\left\langle (0, 0, \ldots, 0)\right\rangle}.
\]
Therefore, it follows from \cite{Ka} that $S(\Lambda^2V)^{\Sp(2k)}$ is also a polynomial algebra with $k$ generators in degrees respectively $1,2, \ldots, k$.
\end{example}

Having in mind that $\OO(2k+1) = \SO(2k+1) \times \{\mathrm{-Id}, \mathrm{Id}\}$, where $\mathrm{Id}$ denotes the identity matrix of the respective dimension, it is clear that $S(\Lambda^2V)^{\OO(2k+1)} = S(\Lambda^2V)^{\SO(2k+1)}$.

We now will determine the structure of $S(\Lambda^2V)^{\OO(2k)}$. For all $n$, we have that $\OO(n)$ is the semidirect product of $\SO(n)$ and the cyclic group $\{R, \mathrm{Id}\}$, where $R$ is a reflection of $\CC^{n}$ which keeps the origin fixed. We first prove the following easy lemma.

\begin{lemma} \label{lemma_O_SO_Inv} If $\dim S^l(W)^{\SO(n)} = 1$ and if $f$ is a generator of $S^l(W)^{\SO(n)}$ then $f$ or $f^2$ is an $\OO(n)$-invariant.
\end{lemma}

\begin{proof}
We fix a reflection $R$ of $\CC^{n}$ so that $\OO(n)$ is the semidirect product of $\SO(n)$ and $\{R, \mathrm{Id}\}$. 
Let $A$ be an element of $\SO(n)$. Then, $ARf = RA_1f$, for some other element $A_1 \in \SO(n)$. Hence, $ARf = Rf$, that is  $Rf \in S^l(W)^{\SO(n)}$. Since $S^l(W)^{\SO(n)}$ is one dimensional and $R$ is a reflection, it follows that $Rf = f$ or $Rf = -f$. This proves the statement.
\end{proof}

\begin{corollary}\label{coro_SLambda2V_O} There exists a generating set $\{f_{2i}\}_{i=1}^{k-1} \cup \{g\}$ of $S(\Lambda^2V)^{\SO(2k)}$ (where $\deg f_j = j$ and $\deg g= k$) such that $\{f_{2i}\}_{i=1}^{k-1} \cup \{g^2\}$ is a generating set for $S(\Lambda^2V)^{\OO(2k)}$. In particular, $S(\Lambda^2V)^{\OO(2k)}$ is a polynomial algebra.
\end{corollary}
\begin{proof}
Let $\{f_{2i}\}_{i=1}^{k-1} \cup \{g\}$ be a generating set for $S(\Lambda^2V)^{\SO(2k)}$.

Assume first that $k$ is odd. Then, Example \ref{ex_SLambda} implies that $\dim S^{2i}(\Lambda^2V)^{\OO(2k)} = \dim S^{2i}(\Lambda^2V)^{\SO(2k)}$ for all $ 0 \leq i \leq k-1$ and that $\dim S^k(\Lambda^2V)^{\SO(2k)} = 1$. Therefore, $S^{2i}(\Lambda^2V)^{\OO(2k)} = S^{2i}(\Lambda^2V)^{\SO(2k)}$ for all $ 0 \leq i \leq k-1$. Furthermore, by Lemma \ref{lemma_O_SO_Inv} and by the fact that $\dim S^k(\Lambda^2V)^{\OO(2k)} = 0$ it follows that $g$ is not $\OO(2k)$-invariant but $g^2$ is an $\OO(2k)$-invariant. Thus, $\{f_{2i}\}_{i=1}^{k-1} \cup \{g^2\}$ is a generating set for $S(\Lambda^2V)^{\OO(2k)}$.

Next, we consider the case when $k = 2j_0$ for some $j_0$. Then, $\dim S^{2i}(\Lambda^2V)^{\OO(2k)} = \dim S^{2i}(\Lambda^2V)^{\SO(2k)}$ for all $ 0 \leq i < j_0$, hence $S^{2i}(\Lambda^2V)^{\OO(2k)} = S^{2i}(\Lambda^2V)^{\SO(2k)}$ for $i < j_0$. In $S^{k}(\Lambda^2V)^{\SO(2k)}$ we have a two dimensional $\OO(2k)$-invariant subspace $U = \mathrm{span} \{f_k, g\}$. Since $R^2 = \mathrm{Id}$, it follows that $R$ has two eigenvalues as a linear operator on $U$, namely $1$ and $-1$. Therefore, $R$ has eigenvectors $\tilde{f}_k$ and $\tilde{g}$ in $U$, such that $R\tilde{f}_k = \tilde{f}_k$ and $R\tilde{g} = -\tilde{g}$. Thus we create a new generating set for $S(\Lambda^2V)^{\SO(2k)}$, by replacing $f_k$ with $\tilde{f}_k$ and $g$ with $\tilde{g}$. Similarly, for $i > j_0$, there is a two dimensional $\OO(2k)$-invariant subspace of $S^{2i}(\Lambda^2V)^{\SO(2k)}$ spanned by $\{f_{2i}, f_{2(i-j_0)} \cdot \tilde{g} \}$. Thus, there is an $\OO(2k)$-invariant vector in $S^{2i}(\Lambda^2V)^{\SO(2k)}$, which we denote by $\tilde{f}_{2i}$. Therefore, $\{f_{2i}\}_{i=1}^{j_0-1} \cup \{\tilde{f}_{2i}\}_{i=j_0}^{k-1} \cup \{\tilde{g}\}$ is a generating set for $S(\Lambda^2V)^{\SO(2k)}$ such that $\{f_{2i}\}_{i=1}^{j_0-1} \cup \{\tilde{f}_{2i}\}_{i=j_0}^{k-1} \cup \{\tilde{g}^2\}$ is a generating set for $S(\Lambda^2V)^{\OO(2k)}$.
\end{proof}

\begin{example}\label{ex_SVLambdaV} Let $W = V \oplus \Lambda^2V$. The decomposition (\ref{eq_PolyAlgDecomp}) of $S(W)$ can be found in the following way
(see also \cite{BBDGK} and, in the language of symmetric functions, \cite[the second edition, page 76, Example 4]{M})

\begin{align}\label{eq_decomp_VLambdaV}
S(V \oplus \Lambda^2V) = S(V) \otimes S(\Lambda^2V) 
= \bigoplus_{\lambda} V^n_{\lambda},
\end{align}
where the last sum is over all partitions $\lambda \in \NN_0^n$. Hence, for the multiplicity series we obtain (see also \cite{BBDGK})
\[
M'(H(S(W)); v_1, \ldots, v_n, t) = \prod_{2i \leq n}\frac{1}{(1-v_{2i-1}t^i)(1-v_{2i}t^i)}, \text{ for }n = 2k,
\]
\[
M'(H(S(W)); v_1, \ldots, v_n, t) = \frac{1}{1-v_nt^{(n+1)/2}}\prod_{2i < n}\frac{1}{(1-v_{2i-1}t^i)(1-v_{2i}t^i)}, \text{ for }n = 2k+1.
\]
Therefore, when $n = 2k$ we obtain
\[
H(S(W)^{\Sp(2k)}, t) = \prod_{i \leq k} \frac{1}{1-t^{i}},
\]
\[
H(S(W)^{\OO(2k)}, t) = \prod_{i \leq k} \frac{1}{(1-t^{2i})^2},
\]
\[
H(S(W)^{\SO(2k)}, t) = \frac{1}{(1-t^{2k})(1-t^k)}\prod_{i \leq k-1} \frac{1}{(1-t^{2i})^2}.
\]
For $n = 2k+1$ we have
\[
H(S(W)^{\OO(2k+1)}, t) = \frac{1}{1-t^{n+1}}\prod_{i \leq k} \frac{1}{(1-t^{2i})^2},
\]
\[
H(S(W)^{\SO(2k+1)}, t) = \frac{1}{1-t^{k+1}}\prod_{i \leq k} \frac{1}{(1-t^{2i})^2}.
\]

It is shown in \cite{S} that $S(W)^{\SO(2k)}$ and $S(W)^{\SO(2k+1)}$ are polynomial algebras and, moreover,
that up to adding trivial summands $W$ is a maximal representation with this property.
(Schwarz calls such representations maximally coregular).
The algebra $S(W)^{\SO(2k)}$ is generated by $2k$ elements, two in each of the degrees $2,4, \ldots, 2(k-1)$
and two more elements -- one in degree $k$ and one in degree $2k$. The algebra $S(W)^{\SO(2k+1)}$ is generated by $2k+1$ elements -- two
in each of the degrees $2,4, \ldots, 2k$ and one generator in degree $k+1$. In more detail, if we use the fact that $S(V \oplus \Lambda^2V) = S(V) \otimes S(\Lambda^2V)$, the degrees of the generators are as follows: (see \cite{S})

For $S(W)^{\SO(2k)}$ we have
\[(0,2), (0,4), \dots, (0, 2k-2); (2,0), (2,2), (2,4), \dots, (2, 2k-2); (0,k).
\]

For $S(W)^{\SO(2k+1)}$ we have
\[(0,2), (0,4), \dots, (0, 2k); (2,0), (2,2), (2,4), \dots, (2, 2k-2); (1,k).
\]


For the polynomiality of $S(W)^{\Sp(2k)}$ we may use \cite{S} again, or we may show it directly as follows.
Corollary \ref{coro_Sp} and the equation (\ref{eq_decomp_VLambdaV}) imply that $S(W)^{\Sp(2k)} = S(\Lambda^2V)^{\Sp(2k)}$.
Therefore, $S(W)^{\Sp(2k)}$ is a polynomial algebra with $k$ generators in degrees $1,2, \ldots, k$.

\end{example}

Example \ref{ex_SVLambdaV} leads to the following corollaries.

\begin{corollary} Let $n = 2k+1$. Let $\{f_{2i}, g_{2i}\}_{i=1}^k \cup \{h\}$ be a generating set for $S(V\oplus \Lambda^2V)^{\SO(2k+1)}$ (where $\deg f_j = (0,j)$, $\deg g_j = (2,j-2)$, $\deg h = (1,k)$). Then $\{f_{2i}, g_{2i}\}_{i=1}^k \cup \{h^2\}$ is a generating set for $S(V\oplus \Lambda^2V)^{\OO(2k+1)}$.
\end{corollary}

\begin{proof}
We will use again that $\OO(2k+1)= \SO(2k+1) \times \{-\mathrm{Id}, \mathrm{Id}\}$. $-\mathrm{Id}$ acts on each generator of $S(V\oplus \Lambda^2V)^{\SO(2k+1)}$ by multiplication with $1$ or $-1$. Since $f_{2i}$ and $g_{2i}$ for all $i \leq k$ are homogeneous polynomials of even degree, it follows that $-\mathrm{Id}$ fixes them. The expressions for the Hilbert series from Example \ref{ex_SVLambdaV} show that $-\mathrm{Id}$ cannot fix $h$, hence $-\mathrm{Id}(h) = -h$, and the statement follows.
\end{proof}

\begin{corollary} Let $n=2k$. There exists a generating set $\{f_{2i}\}_{i=1}^{k-1} \cup \{g_{2i}\}_{i=1}^{k} \cup \{h\}$ of $S(V \oplus \Lambda^2V)^{\SO(2k)}$ (where $\deg f_j = (0,j)$, $\deg g_j = (2,j-2)$, $\deg h = (0,k)$) such that $\{f_{2i}\}_{i=1}^{k-1} \cup \{g_{2i}\}_{i=1}^{k} \cup \{h^2\}$ is a generating set for $S(V \oplus \Lambda^2V)^{\OO(2k)}$. 
\end{corollary}

\begin{proof}
Let $\{f_{2i}\}_{i=1}^{k-1} \cup \{g_{2i}\}_{i=1}^{k} \cup \{h\}$ be a generating set for $S(V \oplus \Lambda^2V)^{\SO(2k)}$. The case when $k$ is odd is the same as in the proof of Corollary \ref{coro_SLambda2V_O} and we skip it.

Let $k = 2j_0$ for some $j_0$. We will use again that $\OO(2k)$ is the semidirect product of $\SO(2k)$ and $\{R, \mathrm{Id}\}$, where $R$ is a reflection in $\CC^{2k}$. As in the proof of Corollary \ref{coro_SLambda2V_O}, for $i < j_0$ we have  $S^{2i}(V \oplus \Lambda^2V)^{\OO(2k)} = S^{2i}(V \oplus \Lambda^2V)^{\SO(2k)}$. We take then
$$
S^k(V \oplus \Lambda^2V)^{\SO(2k)} = \bigoplus_{s+t = j_0}(S^{2s}(V)\otimes S^{2t}(\Lambda^2V))^{\SO(2k)}.
$$

In $(S^2(V)\otimes S^{k-2}(\Lambda^2V))^{\SO(2k)}$ there is a one dimensional $\OO(2k)$-invariant subspace spanned by $g_k$, hence $g_k$ is $\OO(2k)$-invariant. In $(\CC\otimes S^k(\Lambda^2V))^{\SO(2k)}$ there is a two dimensional $\OO(2k)$-invariant subspace $U = \mathrm{span} \{f_k, h\}$. Therefore, there exist vectors $\tilde{f_k}$ and $\tilde{h}$, such that $R\tilde{f_k} = \tilde{f_k}$ and $R\tilde{h} = -\tilde{h}$. Similarly, for $i > j_0$ we have
$$
S^{2i}(V \oplus \Lambda^2V)^{\SO(2k)} = \bigoplus_{s+t = i}(S^{2s}(V)\otimes S^{2t}(\Lambda^2V))^{\SO(2k)}.
$$
In $(S^2(V)\otimes S^{2i-2}(\Lambda^2V))^{\SO(2k)}$ there is a two dimensional $\OO(2k)$-invariant subspace spanned by $\{g_{2i}, \tilde{h}\cdot g_{2i-2j_0}\}$ and therefore there is an $\OO(2k)$-invariant vector which we denote by $\tilde{g_{2i}}$. In $(\CC \otimes S^{2i}(\Lambda^2V))^{\SO(2k)}$ there is also a two dimensional $\OO(2k)$-invariant subspace spanned by $\{f_{2i}, \tilde{h}\cdot f_{2i-2j_0}\}$ and hence an $\OO(2k)$-invariant vector which we denote by $\tilde{f_{2i}}$ .

Thus, as in the proof of Corollary \ref{coro_SLambda2V_O}, we build a new generating set for $S(V \oplus \Lambda^2V)^{\SO(2k)}$ of the form
\[
\{f_{2i}\}_{i=1}^{j_0-1} \cup \{\tilde{f}_{2i}\}_{i=j_0}^{k-1} \cup \{g_{2i}\}_{i=1}^{j_0} \cup \{\tilde{g}_{2i}\}_{i=j_0+1}^{k}\cup \{\tilde{h}\}
\]
such that
\[
\{f_{2i}\}_{i=1}^{j_0-1} \cup \{\tilde{f}_{2i}\}_{i=j_0}^{k-1} \cup \{g_{2i}\}_{i=1}^{j_0} \cup \{\tilde{g}_{2i}\}_{i=j_0+1}^{k} \cup \{\tilde{h^2}\}
\]
is a generating set for $S(V \oplus \Lambda^2V)^{\OO(2k)}$.

\end{proof}

Even if the decomposition (\ref{eq_PolyAlgDecomp}) is not known, for fixed and not very big values of $n$ we can still determine
the Hilbert series of $S(W)^G$ using an algorithm developed in \cite{BBDGK}. This algorithm describes how from the Hilbert series (\ref{eq_HilbSeriesS}) to obtain the multiplicity series $M(H(S(W)), x_1, \dots, x_n,t)$ and $M'(H(S(W)),v_1, \dots, v_n, t)$.
Below we give several explicit examples.

\begin{example}
In this example we take $n = 2$ and $W = S^3V$. Then by \cite{BBDGK}
\begin{align*}
&M(H(S(S^3V)); x_1,x_2, t) = \frac{1 - x_1^2x_2t + x_1^4x_2^2t^2}{(1-x_1^3t)(1-x_1^2x_2t)(1-x_1^6x_2^6t^4)}, \\
&M'(H(S(S^3V)); v_1,v_2, t) = \frac{1 - v_1v_2t + v_1^2v_2^2t^2}{(1-v_1^3t)(1-v_1v_2t)(1-v_2^6t^4)}.
\end{align*}
Therefore, using Theorem \ref{thm_HilbertSeries} we obtain
\[
H(S(S^3V)^{\Sp(2)}, t) = \frac{1}{1-t^4},
\]
\[
H(S(S^3V)^{\OO(2)}, t) = \frac{1}{(1-t^2)^2(1-t^4)},
\]
\[
H(S(S^3V)^{\SO(2)}, t) = \frac{1+t^4}{(1-t^2)^2(1-t^4)}.
\]

The last expression shows that $S(S^3V)^{\SO(2)}$ is not a polynomial algebra.

Since $\Sp(2) = \SL(2)$, it is of course well-known that the algebra $S(S^3V)^{\Sp(2)}$ is a polynomial algebra in one variable
generated by the discriminant of cubic polynomials, i.e., there is a natural choice of the generator in this case.
\end{example}

\begin{example} Let again $n = 2$ and $W = S^4V$. Then in \cite{BBDGK} it is shown that
\[
M(H(S(S^4V)), x_1, x_2, t) = \frac{1-x_1^3x_2t + x_1^6x_2^2t^2}{(1-x_1^4t)(1-x_1^3x_2t)(1-x_1^4x_2^4t^2)(1-x_1^6x_2^6t^3)},
\]
\[
M'(H(S(S^4V)), v_1, v_2, t) = \frac{1-v_1^2v_2t + v_1^4v_2^2t^2}{(1-v_1^4)(1-v_12v_2t)(1-v_2^4t^2)(1-v_2^6t^3)}.
\]
Therefore,
\[
H(S(S^4V)^{\Sp(2)}, t) = \frac{1}{(1-t^2)(1-t^3)},
\]
\[
H(S(S^4V)^{\OO(2)}, t) = \frac{1}{(1-t)(1-t^2)^2(1-t^3)},
\]
\[
H(S(S^4V)^{\SO(2)}, t) = \frac{1+t^3}{(1-t)(1-t^2)^2(1-t^3)}.
\]
As in the previous example, the last expression shows that $S(S^4V)^{\SO(2)}$ is not a polynomial algebra.

It follows from \cite{Ka}, by using again the isomorphism $\Sp(2) = \SL(2)$, that $S(S^4V)^{\Sp(2)}$ is a polynomial algebra
generated by two generators in degrees $2$ and $3$, respectively.
\end{example}

\begin{example} Let $n =2$ and $W = S^2V \oplus S^2V$. Then, by \cite{BBDGK} we have
\[
M(H(S(W)), x_1, x_2, t) = \frac{1+x_1^3x_2t^2}{(1-x_1^2t)^2(1-x_1^2x_2^2t^2)^3};
\]
\[
M'(H(S(W)), v_1, v_2, t) = \frac{1+v_1^2v_2t^2}{(1-v_1^2t)^2(1-v_2^2t^2)^3}.
\]
Therefore,
\begin{align*}
&H(S(W)^{\Sp(2)}, t) = \frac{1}{(1-t^2)^3};\\
&H(S(W)^{\OO(2)}, t) = \frac{1}{(1-t)^2(1-t^2)^3};\\
&H(S(W)^{\SO(2)}, t) = \frac{1+t^2}{(1-t)^2(1-t^2)^3}.
\end{align*}
\end{example}

\begin{example} Let $n=2$ and $W = S^2V \oplus S^3V$. Then, using the formulas for the multiplicity series given in \cite{BBDGK} we obtain that
\begin{align*}
&H(S(W)^{\Sp(2)}, t) = \frac{1+t^7}{(1-t^2)(1-t^3)(1-t^4)(1-t^5)};\\
&H(S(W)^{\OO(2)}, t) = \frac{1+t+t^2+3t^3+2t^4+3t^5+3t^6+3t^7+2t^8+t^9+t^{10}+t^{11}}{(1-t^2)^3(1-t^3)(1-t^4)(1-t^5)};\\
&H(S(W)^{\SO(2)}, t) = \frac{1+t^2+3t^3+4t^4+4t^5+4t^6+3t^7+t^8+t^{10}}{(1-t)(1-t^2)^2(1-t^3)(1-t^4)(1-t^5)}.
\end{align*}
\end{example}

\begin{example} Let $n=2$ and $W = S^3V \oplus S^3V$. Then,
\begin{align*}
&H(S(W)^{\Sp(2)}, t) = \frac{1-t^2 + 2t^4 - t^6 + t^8}{(1-t^2)^2(1-t^4)^3};\\
&H(S(W)^{\OO(2)}, t) = \frac{1+ 2t^2 + 9t^4 + 12t^6 + 12t^8 + 2t^{10}}{(1-t^2)^4(1-t^4)^3};\\
&H(S(W)^{\SO(2)}, t) = \frac{1 + 4t^2 + 21t^4 + 24t^6 + 21t^8 + 4t^{10} + t^{12}}{(1-t^2)^4(1-t^4)^3}.
\end{align*}
\end{example}

\begin{example} Let $n = 3$ and $W = S^3V$. Then using again the formulas for the multiplicity series given in \cite{BBDGK} we obtain that
\[
H(S(S^3V)^{\OO(3)}, t) =\frac{(1+t^4)(1+t^6)(1+t^2+t^4+3t^6+5t^8+3t^{10}+t^{12}+t^{14}+t^{16})}{(1-t^2)(1-t^4)^3(1-t^6)^2(1-t^{10})}.
\]

For the Hilbert series of the algebra $S(S^3V)^{\SO(3)}$ we obtain using again \cite{BBDGK}
\[
H(S(S^3V)^{\SO(3)}, t) =
\frac{t^{14}+t^{13}-2t^{11}+t^9+5t^8+5t^7+5t^6+t^5-2t^3+t+1}{(1-t^3)^2(1-t^5)(1-t^2)^2(1-t^4)^2(1+t)}.
\]
\end{example}

\begin{example} Let again $n=3$ and $W = V^3_{(2,1,0)}$, i.e., $W$ is equal to the irreducible $\GL(3)$-module corresponding to the partition $(2,1,0)$. Then, using \cite{BBDGK} we obtain
\[
H(S(W)^{\OO(3)}, t) = \frac{1+t^4 + t^6 + t^8 + t^{10} + t^{14}}{(1-t^2)^2(1-t^6)^3};
\]
\[
H(S(W)^{\SO(3)}, t) = \frac{1+t^3 + t^4 + t^7 + t^8 + t^{11}}{(1-t^2)^2(1-t^3)(1-t^6)^2}.
\]
\end{example}

\section{The algebra of invariants $\Lambda(S^2V)^G$ for $G = \OO(n), \SO(n), \Sp(2k)$} \label{sec_S}

As a further application of our results and methods developed in Sections \ref{secBR_Sp}, \ref{secBR_OSO}, and \ref{sec_HilbSeries},
in this and the next section we determine the Hilbert series $H(\Lambda(S^2V)^G, t)$ and $H(\Lambda(\Lambda^2V)^G, t)$ for
$G = \OO(n)$, $\SO(n)$, or $\Sp(2k)$. Here again $V = \CC^n$ denotes the standard representation of $\GL(n)$.
Since the modules $\Lambda(S^2V)$ and $\Lambda(\Lambda^2V)$ are finite dimensional,
the Hilbert series of the respective algebras of invariants are polynomials.
We recall that the algebras $\Lambda(S^2V)^{\Sp(2k)}$ and $\Lambda(\Lambda^2V)^{\SO(n)}$ are classically known to be exterior algebras and the degrees of their generators are known. Furthermore, in \cite{I} the algebra $\Lambda(\Lambda^2V)^{\OO(n)}$ is described independently in terms of generators and relations and is shown that it is also isomorphic to an exterior algebra. In our paper, we offer another approach to computing the Hilbert series of $\Lambda(S^2V)^G$ and $\Lambda(\Lambda^2V)^G$ which works both for the known and for the unknown cases. 

For convenience, for the rest of the section we denote $W = \Lambda(S^2 V)$.
The decomposition of $W$ into irreducible $\GL(n)$-modules can be derived using, e.g.,
the formulas in \cite[the second edition, page 79, Example 9 (b)]{M}. We obtain
\[
W = \bigoplus_{\lambda} V_{\lambda}^n,
\]
where the sum runs over all partitions $\lambda = (\alpha_1 + 1, \ldots, \alpha_p +1 \vert \alpha_1, \ldots, \alpha_p)$
in the Frobenius notation with $\alpha_1 \leq n-1$.
If we take into account also the $\NN_0$-grading of $W$, given by the decompostion into irreducible components, we obtain
\[
W = \bigoplus_{i = 0}^{n(n+1)/2} \Lambda^i(S^2V) = \bigoplus_{i = 0}^{n(n+1)/2}  \bigoplus_{\vert\lambda\vert = 2i} V_{\lambda}^n,
\]
where again the sum is over all partitions $\lambda = (\alpha_1 + 1, \ldots, \alpha_p +1 \vert \alpha_1, \ldots, \alpha_p)$
in the Frobenius notation with $\alpha_1 \leq n-1$. Notice that the condition $\alpha_1 \leq n-1$ implies $i \leq n(n+1)/2$.

Using our results from Sections \ref{secBR_Sp}, \ref{secBR_OSO}, and \ref{sec_HilbSeries}, we obtain the following expressions.

\begin{align} \label{eq_HilbSeriesO}
&H(W^{\OO(n)}, t) = \sum_{i\geq 0} \left (\sum_{\substack{\lambda = (\alpha_1 + 1, \ldots, \alpha_p + 1 \vert \alpha_1, \ldots, \alpha_p) \\
\alpha_1 \leq n-1 \text{, } \vert\lambda\vert = 2i \\ \lambda- \text{an even partition}}} 1 \right) t^i;
\end{align}
\begin{align} \label{eq_HilbSeriesSO}
&H(W^{\SO(n)}, t)=\sum_{i\geq 0} \left(\sum_{\substack{\lambda = (\alpha_1 + 1, \ldots, \alpha_p + 1 \vert \alpha_1, \ldots, \alpha_p) \\
\alpha_1 \leq n-1 \text{, }\vert\lambda\vert = 2i \\ \lambda- \text{an even or an odd partition }}} 1 \right) t^i ;
\end{align}
\begin{align} \label{eq_HilbSeriesSp}
&H(W^{\Sp(2k)}, t) = \sum_{i \geq 0} \left(\sum_{\substack{\lambda = (\alpha_1 + 1, \ldots, \alpha_p + 1 \vert \alpha_1, \ldots, \alpha_p) \\
\alpha_1 \leq n-1 \text{, }\vert\lambda\vert = 2i \\ \lambda'- \text{an even partition}}} 1 \right) t^i ,
\end{align}
where $\lambda'$ denotes the transpose partition to $\lambda$.

We recall that if we have a partition $\lambda$ of the form $\lambda = (\alpha_1 + 1, \ldots, \alpha_p + 1 \vert \alpha_1, \ldots, \alpha_p)$
then necessarily $\alpha_1 > \alpha_2 > \cdots > \alpha_p \geq 0$. Therefore partitions of the type
$\lambda = (\alpha_1 + 1, \ldots, \alpha_p + 1 \vert \alpha_1, \ldots, \alpha_p)$ are in one-to-one correspondence with partitions of the form
$\alpha = (\alpha_1, \ldots, \alpha_p)$ with $p$ distinct parts. Moreover, $\vert\lambda\vert = 2\vert\alpha\vert + 2p$.

To determine the above three Hilbert series we fix one $p$, set $X = \{x_1, \ldots, x_p\}$ and define the polynomial
\[
H_p(X, t) = \sum_{i\geq 0} \left (\sum_{\substack{\alpha = (\alpha_1 > \cdots > \alpha_p) \\
\alpha_1 \leq n-1 \text{, }\vert\alpha\vert = i-p}} x_1^{\alpha_1 - (p-1)} x_2^{\alpha_2 - (p-2)} \cdots x_p^{\alpha_p} \right) t^i.
\]
The polynomial $H_p(X,t)$ is in some sence an analogue of the multiplicity series from Section \ref{sec_HilbSeries}.
Notice that since $\alpha$ is a partition with distinct parts, all exponents in the definition of $H_p(X,t)$ are non-negative integers.

We rewrite the above polynomial in the form
\[
H_p(X, t) = \sum_{i\geq 0} x_1^{-(p-1)}x_2^{-(p-2)}\cdots x_{p-1}^{-1} t^p\left (\sum_{\substack{\alpha = (\alpha_1 > \cdots > \alpha_p) \\
\alpha_1 \leq n-1 \text{, }\vert\alpha\vert = i-p}} x_1^{\alpha_1} x_2^{\alpha_2} \cdots x_p^{\alpha_p} \right) t^{i-p}.
\]
Then we set $u_1 = x_1t$, $u_2 = x_1x_2t^2$, \ldots, $u_p = x_1\cdots x_pt^p$ and obtain
\[
H_p(X, t) = \frac{t^{p(p+1)/2}}{u_1 \cdots u_{p-1}}\sum_{i\geq 0} \sum_{\substack{\alpha = (\alpha_1 > \cdots > \alpha_p) \\
\alpha_1 \leq n-1 \text{, }\vert\alpha\vert = i-p}} u_1^{\alpha_1 - \alpha_2} u_2^{\alpha_2-\alpha_3} \cdots u_{p-1}^{\alpha_{p-1} - \alpha_p} u_p^{\alpha_p}.
\]

Now we notice that the polynomial $H_p(X,t)$ is the $(n-p)$-th partial sum of the power series
\[
H_p^{\mathrm{inf}}(X,t) = \frac{t^{p(p+1)/2}}{u_1 \cdots u_{p-1}}\sum_{i\geq 0} \sum_{\substack{\alpha = (\alpha_1 > \cdots > \alpha_p) \\
\vert\alpha\vert = i-p}} u_1^{\alpha_1 - \alpha_2} u_2^{\alpha_2-\alpha_3} \cdots u_{p-1}^{\alpha_{p-1} - \alpha_p} u_p^{\alpha_p}.
\]
For $H_p^{\mathrm{inf}}(X,t)$ we obtain
after some computations
\[
H_p^{\mathrm{inf}}(X, t) = t^{\frac{p(p+1)}{2}} \prod_{k=1}^{p}\frac{1}{1-u_k}.
\]
Using the change of variables $v_1 = x_1$, $v_2 = x_1x_2$, \ldots, $v_p = x_1\cdots x_p$ we have
\[
H_p(X,t)= H_p'(V,t) = \sum_{i\geq 0} \left (\sum_{\substack{\alpha = (\alpha_1 > \cdots > \alpha_p) \\
\alpha_1 \leq n-1 \text{, } \vert\alpha\vert = i-p}}
v_1^{\alpha_1 - \alpha_2-1} v_2^{\alpha_2 - \alpha_3 -1} \cdots v_{p-1}^{\alpha_{p-1} - \alpha_p -1} v_p^{\alpha_p} \right) t^i
\]
and
\[
H_p^{\mathrm{inf}}(X,t)= (H_p')^{\mathrm{inf}}(V,t) = \sum_{i\geq 0} \left (\sum_{\substack{\alpha = (\alpha_1 > \cdots > \alpha_p) \\
\vert\alpha\vert = i-p}}
v_1^{\alpha_1 - \alpha_2-1} v_2^{\alpha_2 - \alpha_3 -1} \cdots v_{p-1}^{\alpha_{p-1} - \alpha_p -1} v_p^{\alpha_p} \right) t^i.
\]
Hence,
\begin{equation} \label{eq_Hp_inf}
(H_p)'^{\mathrm{inf}}(V, t) = t^{\frac{p(p+1)}{2}} \prod_{k=1}^{p}\frac{1}{1-v_kt^k}.
\end{equation}

As we mentioned above, the polynomial $H_p(X,t)$ consists of all terms from $H_p^{\mathrm{inf}}(X,t)$
of the form $t^{\frac{p(p+1)}{2}}u_1^{a_1}\cdots u_p^{a_p}$ such that $a_1 + \cdots + a_p \leq n-p$.
Therefore, using (\ref{eq_Hp_inf}), we obtain that the polynomial $H_p'(V,t)$ consists of all terms
from $(H_p')^{\mathrm{inf}}(V,t)$ of the form $t^{\frac{p(p+1)}{2}}v_1^{a_1} \cdots v_p^{a_p} t^{a_1 + 2a_2 + \cdots + pa_p}$
with $a_1 + \cdots +a_p \leq n-p$. In other words,
\begin{equation} \label{eq_Hp}
H_p'^(V, t) = \sum_{a_1 + \cdots + a_p \leq n-p} v_1^{a_1} \cdots v_p^{a_p} t^{ \frac{p(p+1)}{2} + a_1 + 2a_2 + \cdots + pa_p}.
\end{equation}

Next, we come back to determining the Hilbert series $H(W^{\OO(n)}, t)$. We have the following proposition.

\begin{proposition}
\begin{align*}
H(W^{\OO(n)}, t) =1 + t + &\sum_{\substack{ p=2 \\ p - \text{\rm even}}}^n t^{\frac{p(p+1)}{2}}
\left (\sum_{\substack{a_1 + \cdots + a_{\frac{p}{2}} \leq n-p \\ a_1, \ldots, a_{\frac{p-2}{2}}
- \text{\rm even, } a_{\frac{p}{2}} - \text{\rm odd}}} t^{2a_1 + 4a_2 + \cdots + pa_{\frac{p}{2}}} \right )  \\
+&\sum_{\substack{ p=3 \\
p - \text{\rm odd}}}^n t^{\frac{p(p+1)}{2}} \left(\sum_{\substack{a_1 + \cdots + a_{\frac{p-1}{2}} \leq n-p \\
a_i - \text{\rm even}}} t^{2a_1 + 4a_2 + \cdots + (p-1)a_{\frac{p-1}{2}}}\right).
\end{align*}
\end{proposition}

\begin{proof}
A partition $\lambda = (\alpha_1 + 1, \ldots, \alpha_p + 1 \vert \alpha_1, \ldots, \alpha_p)$ is even if and only if either the following three conditions hold
\begin{itemize}
\item $p$ is even;
\item $\alpha_1, \alpha_3, \ldots, \alpha_{p-1}$ are even;
\item $\alpha_1 - \alpha_2 = 1, \alpha_3 - \alpha_4 = 1, \ldots, \alpha_{p-1} - \alpha_p = 1$.
\end{itemize}
or
\begin{itemize}
\item $p$ is odd;
\item $\alpha_1, \alpha_3, \ldots, \alpha_{p-1}$ are even and $\alpha_p = 0$;
\item $\alpha_1 - \alpha_2 = 1, \alpha_3 - \alpha_4 = 1, \ldots, \alpha_{p-2} - \alpha_{p-1} = 1$.
\end{itemize}
Therefore (\ref{eq_HilbSeriesO}) implies
\[
H(W^{\OO(n)}, t) = \sum_{i\geq 0} \left (\sum_{\substack{p = 0 }}^n
\sum_{\substack{\lambda = (\alpha_1 + 1, \ldots, \alpha_p + 1 \vert \alpha_1, \ldots, \alpha_p) \\
\alpha_1 \leq n-1 \text{, }\vert\lambda\vert = 2i \\ \lambda \text{is an even partition}}} 1 \right ) t^i
= \sum_{i\geq 0} \left (\sum_{\substack{p = 0}}^n\sum_{\substack{\alpha_1 \leq n-1 \text{, }\vert\alpha\vert = i-p }} 1 \right ) t^i,
\]
where the last sum runs over partitions $\alpha = (\alpha_1, \ldots, \alpha_p)$ with distinct parts and such that the above conditions hold.

We rewrite the above as
\[
\sum_{i\geq 0} \left (\sum_{\substack{p = 0}}^n\sum_{\substack{\vert\alpha\vert = i-p \\
\alpha_1 \leq n-1} } 1 \right ) t^i= \sum_{\substack{p = 0}}^n \sum_{i \geq 0} \sum_{\substack{\vert\alpha\vert = i-p \\
\alpha_1 \leq n-1}} t^i.
\]

Next, we fix one even and non-zero $p$ and consider the polynomial $H_p'(V,t)$. We notice that the monomial
$v_1^{\alpha_1 - \alpha_2-1} v_2^{\alpha_2 - \alpha_3 -1} \cdots v_{p-1}^{\alpha_{p-1} - \alpha_p -1} v_p^{\alpha_p}$ evaluated at the point
$(0, v_2, 0, v_4, \ldots, 0, v_p)$ is non-zero if and only if $\alpha_1 - \alpha_2 = 1$, $\alpha_3 - \alpha_4 = 1$, $\ldots$, $\alpha_{p-1} - \alpha_p = 1$.
Therefore we set
\begin{align*}
&M_p(v_2, v_4, \ldots, v_p, t) = H'_p(0, v_2, 0, v_4, \ldots, 0, v_p, t)  \\
&=\sum_{i\geq 0} \left (\sum_{\substack{\alpha = (\alpha_1 > \cdots > \alpha_p) \\
\alpha_1 \leq n-1 \text{, }\vert\alpha\vert = i-p \\
\alpha_1 - \alpha_2 = 1, \alpha_3 - \alpha_4 = 1, \ldots, \alpha_{p-1} - \alpha_p = 1}}
v_2^{\alpha_1 - \alpha_3-2} v_4^{\alpha_3 - \alpha_5 - 2}\cdots v_{p-2}^{\alpha_{p-3} - \alpha_{p -1} - 2} v_p^{\alpha_{p-1} - 1} \right) t^i.
\end{align*}
Next, $\alpha_1$, $\alpha_3$, $\ldots$, $\alpha_{p-1}$ are even numbers if and only if all exponents in the above expression
except the last one are even numbers and $\alpha_{p-1} - 1$ is odd. Therefore we can define iteratively
\begin{align*}
&M_p^{(1)}(v_4, \ldots, v_p, t) = \frac{1}{2}(M_p(1, v_4, \ldots, v_p, t) + M_p(-1, v_4, \ldots, v_p, t)) \\
&\ldots \ldots\\
&M_p^{(p/2-1 )}(v_p, t) = \frac{1}{2}(M_p^{(p/2-2)} (1, v_p, t) + M_p^{(p/2 -2)} (-1, v_p, t)).
\end{align*}
Finally, we define
\[
M_p^{(p/2)}(t) = \frac{1}{2}(M_p^{(p/2-1)} (1, t) - M_p^{(p/2 -1)} (-1,t)).
\]


The next step is to consider the case when $p$ is an odd number and $p > 1$.
Then the monomial $v_1^{\alpha_1 - \alpha_2-1} v_2^{\alpha_2 - \alpha_3 -1} \cdots v_{p-1}^{\alpha_{p-1} - \alpha_p -1} v_p^{\alpha_p}$ evaluated at the point
$(0, v_2, 0, v_4, \ldots, 0, v_{p-1}, 0)$ is non-zero if and only if $\alpha_1 - \alpha_2 = 1$, $\alpha_3 - \alpha_4 = 1$,
$\ldots$, $\alpha_{p-2} - \alpha_{p-1} = 1$ and $\alpha_p = 0$. Therefore we set
\begin{align*}
&N_p(v_2, v_4, \ldots, v_{p-1}, t) = H'_p(0, v_2, 0, v_4, \ldots, 0, v_{p-1}, 0, t)  \\
&=\sum_{i\geq 0} \left (\sum_{\substack{\alpha = (\alpha_1 > \cdots > \alpha_p) \\
\alpha_1 \leq n-1 \text{, }\vert\alpha\vert = i-p \\
\alpha_1 - \alpha_2 = 1, \alpha_3 - \alpha_4 = 1, \ldots, \alpha_{p-2} - \alpha_{p-1} = 1, \alpha_p = 0}}
v_2^{\alpha_1 - \alpha_3-2} v_4^{\alpha_3 - \alpha_5-2}\cdots v_{p-1}^{\alpha_{p-2} - \alpha_{p} -2} \right) t^i.
\end{align*}
We notice again that $\alpha_1$, $\alpha_3$, $\ldots$, $\alpha_{p}$ are even numbers
if and only if all exponents in the above expression are even numbers. Hence, similarly to the previous case we can define iteratively
\begin{align*}
&N_p^{(1)}(v_4, \ldots, v_{p-1}, t) = \frac{1}{2}(N_p(1, v_4, \ldots, v_{p-1}, t) + N_p(-1, v_4, \ldots, v_{p-1}, t)) \\
&\ldots \ldots\\
&N_p^{(\frac{p-1}{2})}(t) = \frac{1}{2}(N_p^{(\frac{p-1}{2}-1)} (1, t) + N_p^{(\frac{p-1}{2} -1)} (-1,t)).
\end{align*}


Finally, we consider the cases $p=0$ and $p=1$. For $p = 0$ we set $H'_0(V,t) = 1$. For $p=1$ we have
\[
H'_1(v_1,t) = \sum_{\alpha_1 = 0}^{n-1} v_1^{\alpha_1} t^{\alpha_1 + 1}.
\]
Since the partition $\lambda = (\alpha_1 + 1 \vert \alpha_1)$ is even if and only if $\alpha_1 = 0$,
we evaluate $H'_1(v_1,t)$ at the point $(0, t)$ and obtain $H'_1(0,t) = t$.


The statement of the proposition follows now from (\ref{eq_Hp}) and the observation that
\[
H(W^{\OO(n)}, t) = H'_0(V,t) + H'_1(0,t) + \sum_{\substack{p = 2 \\ p - \text{even}}}^n M_p^{(p/2)}(t)
+ \sum_{\substack{p = 3 \\ p - \text{odd}}}^n N_p^{(\frac{p-1}{2})}(t).
\]
\end{proof}

We determine the Hilbert series $H(W^{\Sp(2k)}, t)$ and $H(W^{\SO(n)}, t)$ in a similar way
by using respectively (\ref{eq_HilbSeriesSp}) and (\ref{eq_HilbSeriesSO}).
To determine $H(W^{\Sp(2k)}, t)$ we notice that for a partition $\lambda = (\alpha_1 +1, \ldots,\alpha_p +1 \vert \alpha_1, \ldots, \alpha_p)$
the transpose $\lambda'$ is even if and only if
\begin{itemize}
\item $p$ is even;
\item $\alpha_1, \alpha_3, \ldots, \alpha_{p-1}$ are odd numbers;
\item $\alpha_1 - \alpha_2 = 1, \alpha_3 - \alpha_4 = 1, \ldots, \alpha_{p-1} - \alpha_p = 1$.
\end{itemize}
Therefore, by fixing $p$ even and evaluating the series $H'_p(V,t)$ at well-chosen points we obtain:

\begin{proposition} \label{prop_HilbSeries_S_Sp}
Let $n = 2k$. Then
\[
H(W^{\Sp(2k)}, t) = 1 + \sum_{\substack{ p=2 \\ p - \text{\rm even}}}^n t^{\frac{p(p+1)}{2}}
\left (\sum_{\substack{a_1 + \cdots + a_{\frac{p}{2}} \leq n-p \\
a_i - \text{\rm even}}} t^{2a_1 + 4a_2 + \cdots + pa_{\frac{p}{2}}} \right ).
\]
\end{proposition}


It remains to determine the Hilbert series $H(W^{\SO(n)}, t)$.
First we notice that a partition $\lambda = (\alpha_1 +1, \ldots,\alpha_p +1 \vert \alpha_1, \ldots, \alpha_p)$
is odd if and only if either the following four conditions hold
\begin{itemize}
\item $n$ is even;
\item $p$ is odd;
\item $\alpha_1, \alpha_3, \ldots, \alpha_{p-2}, \alpha_p$ are odd numbers and $\alpha_1 = n-1$;
\item $\alpha_2 - \alpha_3 = 1, \alpha_4 - \alpha_5 = 1, \ldots, \alpha_{p-1} - \alpha_p = 1$.
\end{itemize}
 or
\begin{itemize}
\item $n$ is even;
\item $p$ is even;
\item $\alpha_1, \alpha_3, \ldots, \alpha_{p-1}$ are odd numbers and $\alpha_1 = n-1$;
\item $\alpha_2 - \alpha_3 = 1, \alpha_4 - \alpha_5 = 1, \ldots, \alpha_{p-2} - \alpha_{p-1} = 1$ and $\alpha_p = 0$.
\end{itemize}
Then, we fix $p$ and evaluate the series $H'_p(V,t)$ from (\ref{eq_Hp}) at well-chosen points to obtain:
\begin{proposition}
\begin{itemize}
\item [(i)] Let $n = 2k+1$. Then
\[
H(W^{\SO(n)}, t) = H(W^{\OO(n)}, t).
\]
\item[(ii)] Let $n = 2k$. Then
\begin{align*}
H(W^{\SO(n)}, t) =1 + t + &\sum_{\substack{ p=2 \\ p - \text{\rm even}}}^n t^{\frac{p(p+1)}{2}}
\left (\sum_{\substack{a_1 + \cdots + a_{\frac{p}{2}} \leq n-p \\ a_1, \ldots, a_{\frac{p-2}{2}} - \text{\rm even, }
a_{\frac{p}{2}} - \text{\rm odd}}} t^{2a_1 + 4a_2 + \cdots + pa_{\frac{p}{2}}} \right ) \\
+&\sum_{\substack{ p=3 \\ p - \text{\rm odd}}}^n t^{\frac{p(p+1)}{2}} \left(\sum_{\substack{a_1 + \cdots + a_{\frac{p-1}{2}} \leq n-p \\
a_i - \text{\rm even}}} t^{2a_1 + 4a_2 + \cdots + (p-1)a_{\frac{p-1}{2}}} \right) \\
+&\sum_{\substack{ p=1 \\ p - \text{odd}}}^n t^{\frac{p(p+1)}{2}} \left (\sum_{\substack{a_1 + \cdots + a_{\frac{p+1}{2}} = n-p \\
a_1, \ldots, a_{\frac{p-1}{2}} - \text{\rm even, } a_{\frac{p+1}{2}} - \text{\rm odd}}} t^{a_1 + 3a_2 + \cdots + pa_{\frac{p+1}{2}}} \right) \\
+&\sum_{\substack{ p=2 \\ p - \text{\rm even}}}^n t^{\frac{p(p+1)}{2}} \left(\sum_{\substack{a_1 + \cdots + a_{\frac{p}{2}} = n-p \\
a_i - \text{\rm even}}} t^{a_1 + 3a_2 + \cdots + (p-1)a_{\frac{p}{2}}} \right).
\end{align*}
\end{itemize}
\end{proposition}

\section{The algebra of invariants $\Lambda(\Lambda^2 V)^G$ for $G = \OO(n), \SO(n), \Sp(2k)$} \label{sec_Lambda}

In this section we set $W = \Lambda(\Lambda^2 V)$. The approach for computing the Hilbert series $H(\Lambda(\Lambda^2V)^G, t)$
is very similar to the one in the previous section and we shall only sketch the proofs.


The decomposition of $W$ into irreducible $\GL(n)$-modules can again be determined using, e.g., the formulas from
\cite[pages 78-79, Example 9 (a)]{M}. The exact formula is
\[
W = \bigoplus_{\lambda} V^n_{\lambda} = \bigoplus_{i =0}^{n(n-1)/2} \bigoplus_{\vert\lambda\vert = 2i} V^n_{\lambda},
\]
where the sum runs over all partitions $\lambda = (\alpha_1 - 1, \ldots, \alpha_p - 1 \vert \alpha_1, \ldots, \alpha_p)$
in the Frobenius notation and $\alpha_1 \leq n -1$. Then for the Hilbert series of the respective algebras of invariants we obtain the following expressions:

\begin{align*} 
&H(W^{\OO(n)}, t) = \sum_{i\geq 0} \left (\sum_{\substack{\lambda = (\alpha_1 - 1, \ldots, \alpha_p - 1 \vert \alpha_1, \ldots, \alpha_p) \\
\alpha_1 \leq n-1 \text{, } \vert\lambda\vert = 2i \\ \lambda- \text{an even partition}}} 1 \right) t^i;
\end{align*}
\begin{align*} 
&H(W^{\SO(n)}, t)=\sum_{i\geq 0} \left(\sum_{\substack{\lambda = (\alpha_1 - 1, \ldots, \alpha_p - 1 \vert \alpha_1, \ldots, \alpha_p) \\
\alpha_1 \leq n-1 \text{, }\vert\lambda\vert = 2i \\ \lambda- \text{an even or an odd partition }}} 1 \right) t^i ;
\end{align*}
\begin{align*} 
&H(W^{\Sp(2k)}, t) = \sum_{i \geq 0} \left(\sum_{\substack{\lambda = (\alpha_1 - 1, \ldots, \alpha_p - 1 \vert \alpha_1, \ldots, \alpha_p) \\
\alpha_1 \leq n-1 \text{, }\vert\lambda\vert = 2i \\ \lambda'- \text{an even partition}}} 1 \right) t^i ,
\end{align*}
where $\lambda'$ denotes the transpose partition to $\lambda$.

We notice that partitions of the type
$\lambda = (\alpha_1 - 1, \ldots, \alpha_p - 1 \vert \alpha_1, \ldots, \alpha_p)$ are in one-to-one correspondence with partitions of the form
$\alpha = (\alpha_1, \ldots, \alpha_p)$ with $p$ distinct positive parts. Moreover $\vert\lambda\vert = 2\vert\alpha\vert$.

As in the previous section we fix $p$, set $X = \{x_1, \ldots, x_p\}$ and define the following analogue of the multiplicity series
\[
H_p(X, t) = \sum_{i\geq 0} \left (\sum_{\substack{\alpha = (\alpha_1 > \cdots > \alpha_p > 0) \\
\alpha_1 \leq n-1 \text{, }\vert\alpha\vert = i}} x_1^{\alpha_1 - p} x_2^{\alpha_2 - (p-1)} \cdots x_p^{\alpha_p - 1} \right) t^i.
\]
We make the same transformations as in the previous section.
\[
H_p(X, t) = \sum_{i\geq 0} x_1^{-p}x_2^{-(p-1)}\cdots x_{p}^{-1}\left (\sum_{\substack{\alpha = (\alpha_1 > \cdots > \alpha_p>0) \\
\alpha_1 \leq n-1 \text{, }\vert\alpha\vert = i}} x_1^{\alpha_1} x_2^{\alpha_2} \cdots x_p^{\alpha_p} \right) t^i.
\]
Then we set again $u_1 = x_1t$, $u_2 = x_1x_2t^2$, \ldots, $u_p = x_1\cdots x_pt^p$ and obtain
\[
H_p(X, t) = \frac{t^{p(p+1)/2}}{u_1 \cdots u_{p}}\sum_{i\geq 0} \sum_{\substack{\alpha = (\alpha_1 > \cdots > \alpha_p>0) \\
\alpha_1 \leq n-1 \text{, }\vert\alpha\vert = i}} u_1^{\alpha_1 - \alpha_2} u_2^{\alpha_2-\alpha_3} \cdots u_{p-1}^{\alpha_{p-1} - \alpha_p} u_p^{\alpha_p}.
\]

Now we notice that the polynomial $H_p(X,t)$ is the $(n-p-1)$-st partial sum of the power series
\[
H_p^{\mathrm{inf}}(X,t) = \frac{t^{p(p+1)/2}}{u_1 \cdots u_{p}}\sum_{i\geq 0} \sum_{\substack{\alpha = (\alpha_1 > \cdots > \alpha_p>0) \\
\vert\alpha\vert = i}} u_1^{\alpha_1 - \alpha_2} u_2^{\alpha_2-\alpha_3} \cdots u_{p-1}^{\alpha_{p-1} - \alpha_p} u_p^{\alpha_p}.
\]
We derive after some computations that
\[
H_p^{\mathrm{inf}}(X, t) = t^{\frac{p(p+1)}{2}} \prod_{k=1}^{p}\frac{1}{1-u_k}.
\]
Using the change of variables $v_1 = x_1$, $v_2 = x_1x_2$, \ldots, $v_p = x_1\cdots x_p$ we obtain
\[
H_p(X,t)= H_p'(V,t) = \sum_{i\geq 0} \left (\sum_{\substack{\alpha = (\alpha_1 > \cdots > \alpha_p>0) \\
\alpha_1 \leq n-1 \text{, } \vert\alpha\vert = i}}
v_1^{\alpha_1 - \alpha_2-1} v_2^{\alpha_2 - \alpha_3 -1} \cdots v_{p-1}^{\alpha_{p-1} - \alpha_p -1} v_p^{\alpha_p - 1} \right) t^i
\]
and
\[
H_p^{\mathrm{inf}}(X,t)= (H_p')^{\mathrm{inf}}(V,t) = \sum_{i\geq 0} \left (\sum_{\substack{\alpha = (\alpha_1 > \cdots > \alpha_p>0) \\
\vert\alpha\vert = i}}
v_1^{\alpha_1 - \alpha_2-1} v_2^{\alpha_2 - \alpha_3 -1} \cdots v_{p-1}^{\alpha_{p-1} - \alpha_p -1} v_p^{\alpha_p - 1} \right) t^i.
\]
Therefore,
\begin{align*}
(H_p')^{\mathrm{inf}}(V, t) = t^{\frac{p(p+1)}{2}} \prod_{k=1}^{p}\frac{1}{1-v_kt^k}.
\end{align*}

The polynomial $H_p(X,t)$ consists of all terms from $H_p^{\mathrm{inf}}(X,t)$ of the form $t^{\frac{p(p+1)}{2}}u_1^{a_1}\cdots u_p^{a_p}$
such that $a_1 + \cdots + a_p \leq n-p -1$.
Therefore, the polynomial $H_p'(V,t)$ consists of all terms from $(H_p')^{\mathrm{inf}}(V,t)$
of the form $t^{\frac{p(p+1)}{2}}v_1^{a_1} \cdots v_p^{a_p} t^{a_1 + 2a_2 + \cdots + pa_p}$ with $a_1 + \cdots +a_p \leq n-p -1$. In other words,
\begin{equation} \label{eq_Hp_Lambda}
H_p'(V, t) = \sum_{a_1 + \cdots + a_p \leq n-p-1} v_1^{a_1} \cdots v_p^{a_p} t^{ \frac{p(p+1)}{2} + a_1 + 2a_2 + \cdots + pa_p}.
\end{equation}

The following propositions now hold.

\begin{proposition} \label{prop_HilbSeries_Lambda_O}
\begin{align*}
H(W^{\OO(n)}, t) =1  + &\sum_{\substack{ p=2 \\ p  - \text{\rm even}}}^{n-1} t^{\frac{p(p+1)}{2}}
\left (\sum_{\substack{a_1 + \cdots + a_{\frac{p}{2}} \leq n-p-1 \\ a_1, \ldots, a_{\frac{p}{2}} - \text{\rm even}}}
t^{2a_1 + 4a_2 + \ldots + pa_{\frac{p}{2}}} \right ).
\end{align*}
\end{proposition}

\begin{proof}
A partition $\lambda = (\alpha_1 - 1, \ldots, \alpha_p - 1 \vert \alpha_1, \ldots, \alpha_p)$ is even if and only if
\begin{itemize}
\item $p$ is even;
\item $\alpha_1, \alpha_3, \ldots, \alpha_{p-1}$ are even;
\item $\alpha_1 - \alpha_2 = 1, \alpha_3 - \alpha_4 = 1, \ldots, \alpha_{p-1} - \alpha_p = 1$.
\end{itemize}

Therefore, fixing $p$ even and evaluating the polynomial $H'_p(V,t)$ in (\ref{eq_Hp_Lambda}) at well-chosen points we obtan the desired result.
\end{proof}


Next, we notice that Proposition \ref{prop_HilbSeries_S_Sp} and Proposition \ref{prop_HilbSeries_Lambda_O} imply the following corollary.

\begin{corollary} \label{coro_Equality}
\[
H(\Lambda(S^2V)^{\Sp(2k)}, t) = H(\Lambda(\Lambda^2V)^{\OO(2k+1)}, t).
\]
\end{corollary}

It remains to consider the algebras $\Lambda(\Lambda^2V)^{\Sp(2k)}$ and $\Lambda(\Lambda^2V)^{\SO(n)}$.
\begin{proposition} \label{prop_Hilb_Lambda_Sp}
Let $n = 2k$. Then
\begin{align*}
H(W^{\Sp(2k)}, t) = 1 + t + &\sum_{\substack{ p=2 \\ p - \text{\rm even}}}^{n-1} t^{\frac{p(p+1)}{2}}
\left (\sum_{\substack{a_1 + \cdots + a_{\frac{p}{2}} \leq n-p -1 \\
a_1, \ldots, a_{\frac{p-2}{2}} - \text{\rm even, } a_{\frac{p}{2}} - \text{\rm odd}}} t^{2a_1 + 4a_2 + \cdots + pa_{\frac{p}{2}}} \right ) + \\
&\sum_{\substack{ p=3 \\ p - \text{\rm odd}}}^{n-1} t^{\frac{p(p+1)}{2}}
\left (\sum_{\substack{a_1 + \cdots + a_{\frac{p-1}{2}} \leq n-p -1 \\
a_1, \ldots, a_{\frac{p-1}{2}} - \text{\rm even}}} t^{2a_1 + 4a_2 + \cdots + (p-1)a_{\frac{p-1}{2}}} \right).
\end{align*}
\end{proposition}

\begin{proof}
For a partition $\lambda = (\alpha_1 -1, \ldots,\alpha_p -1 \vert \alpha_1, \ldots, \alpha_p)$
the transpose $\lambda'$ is even if and only if either the following three conditions hold
\begin{itemize}
\item $p$ is even;
\item $\alpha_1, \alpha_3, \ldots, \alpha_{p-1}$ are odd numbers;
\item $\alpha_1 - \alpha_2 = 1, \alpha_3 - \alpha_4 = 1, \ldots, \alpha_{p-1} - \alpha_p = 1$;
\end{itemize}
or
\begin{itemize}
\item $p$ is odd;
\item $\alpha_1, \alpha_3, \ldots, \alpha_{p}$ are odd numbers and $\alpha_p = 1$;
\item $\alpha_1 - \alpha_2 = 1, \alpha_3 - \alpha_4 = 1, \ldots, \alpha_{p-2} - \alpha_{p-1} = 1$.
\end{itemize}
Therefore, by fixing $p$ and evaluating the series $H'_p(V,t)$ from (\ref{eq_Hp_Lambda}) at well-chosen points we obtain the statement of the proposition.
\end{proof}

\begin{corollary}
\[
H(\Lambda(\Lambda^2V)^{\Sp(2k)}, t) = H(\Lambda(S^2V)^{\OO(2k-1)}, t) = H(\Lambda(S^2V)^{\SO(2k-1)}, t).
\]
\end{corollary}

\begin{proposition} \label{prop_Hilb_Lambda_SO}
\begin{itemize}
\item [(i)] Let $n=2k+1$. Then
\begin{align*}
 H(W^{\SO(n)}, t) = H(W^{\OO(n)}, t) = 1 + &\sum_{\substack{ p=2 \\ p - \text{\rm even}}}^{n-1} t^{\frac{p(p+1)}{2}}
 \left (\sum_{\substack{a_1 + \cdots + a_{\frac{p}{2}} \leq n-p-1 \\
a_1, \ldots, a_{\frac{p}{2}} - \text{\rm even}}} t^{2a_1 + 4a_2 + \cdots + pa_{\frac{p}{2}}} \right ).
\end{align*}
\item [(ii)] Let $n = 2k$. Then
\begin{align*}
H(W^{\SO(n)}, t) =1 + &\sum_{\substack{ p=2 \\ p - \text{\rm even}}}^{n-1} t^{\frac{p(p+1)}{2}} \left (\sum_{\substack{a_1 + \cdots + a_{\frac{p}{2}} \leq n-p-1 \\
a_1, \ldots, a_{\frac{p}{2}} - \text{\rm even}}} t^{2a_1 + 4a_2 + \cdots + pa_{\frac{p}{2}}} \right ) \\
+&\sum_{\substack{ p=1 \\ p - \text{\rm odd}}}^{n-1} t^{\frac{p(p+1)}{2}} \left(\sum_{\substack{a_1 + \cdots + a_{\frac{p+1}{2}} = n-p-1 \\
a_i - \text{\rm even}}} t^{a_1 + 3a_2 + \cdots + pa_{\frac{p+1}{2}}} \right).
\end{align*}
\end{itemize}

\end{proposition}

\begin{proof}
A partition $\lambda = (\alpha_1 -1, \ldots,\alpha_p -1 \vert \alpha_1, \ldots, \alpha_p)$ is odd if and only if the following conditions hold
\begin{itemize}
\item $n$ is even;
\item $p$ is odd;
\item $\alpha_1, \alpha_3, \ldots, \alpha_{p-2}, \alpha_p$ are odd numbers and $\alpha_1 = n-1$;
\item $\alpha_2 - \alpha_3 = 1, \alpha_4 - \alpha_5 = 1, \ldots, \alpha_{p-1} - \alpha_p = 1$.
\end{itemize}

Then, we fix $p$ odd and evaluate the series $H'_p(V,t)$ from (\ref{eq_Hp_Lambda}) at well-chosen points to obtain the result.
\end{proof}


\begin{thebibliography} {99}

\bibitem{BBDGK}
F. Benanti, S. Boumova, V. Drensky, G. K.  Genov, P. Koev,
{\it Computing with rational symmetric functions and applications to invariant theory and PI-algebras},
Serdica Math. J. {\bf 38} (2012), Nos 1-3, 137-188.

\bibitem{B}
A. Berele,
{\it Applications of Belov's theorem to the cocharacter
sequence of p.i. algebras},
J. Algebra {\bf 298} (2006), 208-214.

\bibitem{DK}
H. Derksen, G. Kemper,
{\it Computational Invariant Theory},
Invariant Theory and Algebraic Transformation Groups, I.
Encyclopaedia of Mathematical Sciences, {\bf 130},
Springer-Verlag, Berlin, 2002.

\bibitem{DG}
V. Drensky, G. K. Genov,
{\it Multiplicities of Schur functions in invariants of two $3 \times 3$ matrices}, J. Algebra {\bf 264} (2003), 496-519.

\bibitem{E}
E. B. Elliott,
{\it On linear homogeneous diophantine equations},
Quart. J. Pure Appl. Math. {\bf 34} (1903), 348-377.

\bibitem{FH}
W. Fulton, J. Harris,
{\it Representation Theory. A First Course}.
Graduate Texts in Mathematics, {\bf 129}, Readings in Mathematics, Springer-Verlag, New York, 1991.

\bibitem{GW}
R. Goodman, N. Wallach,
{\it Symmetry, Representations, and Invariants},
Graduate Texts in Mathematics, {\bf 255},
Springer-Verlag, Dordrecht--Heidelberg--London--New York, 2009.

\bibitem{HTW}
R. Howe, E. Tan, J. Willenbring,
{\it Stable branching rules for classical symmetric pairs},
Trans. Amer. Math. Soc. {\bf 357} (2005), 1601-1626.

\bibitem{I}
M. Itoh,
{\it Invariant theory in exterior algebras and Amitsur-Levitzki type theorems},
Adv. Math. {\bf 288} (2016), 679-701.

\bibitem{Ka}
V. Kac,
{\it Some remarks on nilpotent orbits},
J. Algebra {\bf 64} (1980), 190-213.

\bibitem{K}
R. King,
{\it Modification rules and products of irreducible representations for the unitary, orthogonal, and symplectic groups},
J. Math. Phys. {\bf 12} (1971), 1588-1598.

\bibitem{M}
I. G. Macdonald,
{\it Symmetric Functions and Hall Polynomials},
Oxford Mathematical Monographs, The Clarendon Press, Oxford University Press,
New York, 1979, Second Edition, 1995.


\bibitem{MM}
P. A. MacMahon,
{\it Combinatory Analysis}, vols 1 and 2,
Cambridge Univ. Press. 1915, 1916.
Reprinted in one volume:
Chelsea, New York, 1960.

\bibitem{P}
C. Procesi,
{\it Lie Groups, an approach through Invariants and Representations},
Universitex, Springer.


\bibitem{S} G. W. Schwarz,
{Representations of simple Lie groups with regular rings of invariants},
Invent. Math. {\bf 49} (1978), 167-191.
\end{thebibliography}
\end{document}